\documentclass[12pt]{amsart}
\usepackage{amsthm,amsfonts,amsmath,amssymb,latexsym,epsfig,mathrsfs,yfonts,marvosym}
\usepackage{graphicx,color}
\usepackage{url}
\usepackage[colorlinks=true]{hyperref}
\usepackage{amssymb}

\usepackage[margin=0.835in]{geometry}

\newtheorem{lemma}{Lemma}[section]
 
\newtheorem{remark}{Remark}[section]
\newtheorem{theorem}{Theorem}
\newtheorem{corollary}{Corollary}
\newtheorem{example}{Example}
\newtheorem{definition}{Definition}

\newtheorem{proposition}{Proposition}

\newtheorem*{ack}{Acknowledgements}

\def\<{\langle}

\def\>{\rangle}

\def\b{\beta}

\newcommand{\grad}{\mathop{\mathrm{grad}}\nolimits}

\def\bbc{{\mathbb C}}

\def\bbf{{\mathbb F}}

\def\bbn{{\mathbb N}}

\def\bbp{{\mathbb P}}
\def\bbq{{\mathbb Q}}
\def\bbr{{\mathbb R}}

\def\bbz{{\mathbb Z}}

\def\bfw{{\bf w}}

\def\bfH{{\bf H}}

\def\cald{{\mathcal D}}

\def\call{{\mathcal L}}

\def\calo{{\mathcal O}}

\def\gt{{\mathfrak t}}

\def\gz{{\mathfrak z}}

\def\lra{\longrightarrow}

\def\<{\langle}
\def\>{\rangle}
\def\ra {\rightarrow}

\def\b{\beta}

\newcommand{\hirz}{\mathfrak{n}}

\def\kt{\mathfrak{t}}

\def\hook{\mathbin{\hbox to 6pt{%
                 \vrule height0.4pt width5pt depth0pt
                 \kern-.4pt
                 \vrule height6pt width0.4pt depth0pt\hss}}}

\begin{document}

\title{Twins in K\"ahler and Sasaki geometry}
\author[C.P. Boyer]{Charles P. Boyer}
\author[H. Huang]{ Hongnian Huang}
\author[E. Legendre]{Eveline Legendre}
\author[C.W. T{\o}nnesen-Friedman]{Christina W. T{\o}nnesen-Friedman}

\address{Charles P. Boyer, Department of Mathematics and Statistics,
University of New Mexico, Albuquerque, NM 87131.}
\email{cboyer@unm.edu} 
\address{Hongnian Huang, Department of Mathematics and Statistics,
University of New Mexico, Albuquerque, NM 87131.}
\email{hnhuang@gmail.com} 
\address{Eveline Legendre\\ Universit\'e Claude Bernard Lyon 1\\
Institut Camille Jordan \'equipe AGL\\ 21 av. Claude Bernard\\
69100 Villeurbanne\\ France}
\email{eveline.legendre@math.univ-lyon1.fr}
\address{Christina W. T{\o}nnesen-Friedman, Department of Mathematics, Union
College, Schenectady, New York 12308, USA } \email{tonnesec@union.edu}

\thanks{The first author was partially supported by grant \#519432 from the Simons Foundation. The third author is partially supported by France ANR project BRIDGES No ANR-21-CE40-0017. }

\begin{abstract}
We introduce the notions of weighted extremal K\"ahler twins together with the related notion of extremal Sasaki twins. In the K\"ahler setting this leads to a generalization of the twinning phenomenon appearing among LeBrun's strongly Hermitian solutions to the Einstein-Maxwell equations on the first Hirzebruch surface \cite{Leb16} to weighted extremal metrics on Hirzebruch surfaces in general. We discover that many twins appear and that this can
be viewed in the Sasaki setting as a case where we have more than one extremal ray in the Sasaki cone even when we do not allow changes within the isotopy class. We also study extremal Sasaki twins directly in the Sasaki setting with a main focus on the toric Sasaki case.
\end{abstract}

\date{\today}

\maketitle

\markboth{Weighted Extremal Twins}{C. Boyer, H. Huang, E. Legendre and C. T{\o}nnesen-Friedman}

{
  \hypersetup{linkcolor=black}
  \tableofcontents
}

\section{Introduction}\label{intro}

The purpose of this paper is to introduce and explore the existence of what we call {\it weighted extremal K\"ahler twins} in which a K\"ahler metric is weighted extremal in the sense of Apostolov-Calderbank \cite{ApCa18} with respect to some fixed weight $p$ but with respect to two different smooth positive Killing potentials that are not scalar multiples of each other. More precisely, recall that for a given $p\in {\mathbb R}$, a Kähler manifold $(N^n,J,g,\omega)$ with a fixed Killing potential $f$ (i.e $\call_{\nabla^gf}J=0$) is said to be a {\it $(f,p)$-extremal K\"ahler structure} if the $(f,p)$-\emph{scalar curvature}, that is
\begin{equation}\label{weightedscal}
Scal_{f,p}(g):= f^2 Scal(g)  -2(p-1) f\Delta_g f - p(p-1)|df|^2_g,
\end{equation} is a Killing potential, \cite{ApCa18}.  This is a direct extension of the now classical extremal K\"ahler metrics of Calabi \cite{cal82} which correspond to the case $(f,p)=(1,p)$ for any $p$. A natural question arises: {\it Is it possible to have a K\"ahler metric which is $(f,p)$-extremal with respect to more than one positive Killing potential, $f$, which are not rescales of each other?} As we shall see in this paper, the answer is a resounding ``Yes'' and the {\it twin} structures obtained this way encompass many well known K\"ahler and Hermitian surfaces like the Page metric \cite{Pag} and many other examples eg. \cite{Leb10,Leb15}.\\

We thus introduce the notion of {\it weighted extremal K\"ahler twins} as follows.   

\begin{definition}\label{twinsKahler}
Let $(N,J,g,\omega)$ be a K\"ahler manifold of complex dimension $n$ and let $p\in \bbr$. Suppose that $f_1,f_2$ are two positive, smooth Killing potentials on $(N,J,g,\omega)$ that are not constant rescales of each other. 
If $g$ is simultaneously $(f_1,p)$-extremal and $(f_2,p)$ extremal then we say that $(g,f_1)$ and $(g,f_2)$ are $p$-\emph{weighted extremal K\"ahler twins}. 
\end{definition}
Our interest in weighted extremal K\"ahler twins comes essentially from two geometric cases where each twin corresponds to a genuine "metric structure" either 
Hermitian or Sasakian:
\begin{itemize}
\item When $p=2n$, $Scal_{f,2n}(g)$ computes the scalar curvature of the Hermitian metric $h = f^{-2} g$. In particular, the case $n=2$ and $p=4$ is of great interest due
to a discovery by LeBrun. 
If $Scal_{f,4}(g)$ is constant, this corresponds to the strongly Hermitian solutions of the Einstein-Maxwell equations \cite{Leb10,Leb15} 
\begin{equation}\label{EM}
dF=0, \qquad d\star F=0, \qquad ({\rm Ric}_h +F\circ F)_0=0
\end{equation}
where $F$ is a smooth 2-form and the subscript $0$ denotes the trace free component. Hence, $F$ is just an harmonic 2-form that satisfies the 3rd equation. On the trivial and first Hirzebruch surfaces, LeBrun found many such solutions (see \cite{Leb15, Leb16}) among which the Page metric \cite{Pag} on the first Hirzebruch surface is a special case.

\item  When $p=n+2$ and $(J,g,\omega)$ is the transversal K\"ahler structure of a $(2n+1)$-dimensional Sasaki manifold $(M, \cald, J,\xi_1)$ in the sense that $\omega=d\eta$ where $\eta$ is a contact $1$-form with Reeb vector field $\xi_1$ and such that $(\cald=\ker \eta, J)$ is a CR structure, then each $\xi_1$-invariant positive Killing potential $f$ determines a new Sasaki structure $(M, \cald, J,\xi_f)$ with the same underlying CR-structure but another Reeb vector field, $\xi_f$ such that $\xi_f$ and $\xi_1$ commute. We refer to \cite{ApCa18} where it is proven that $g$ is a {\it $(f,n+2)$-extremal K\"ahler metric} if and only if $\xi_f$ is the Reeb vector field of an extremal Sasaki structure in the sense of \cite{BGS06}. 

\end{itemize} 

This last situation leads to the following natural notion. 
\begin{definition}\label{twinsSasaki}
Let $(M,\cald, J)$ be a pseudo-convex CR structure of Sasaki type on a compact smooth manifold $M$ of dimension $2n+1$. \emph{Extremal Sasaki twins} of $(\cald,J)$ are two extremal Sasaki structures both compatible with $(\cald,J)$ and having commuting non-colinear Sasaki-Reeb vector fields. 
\end{definition}

Note that for a compact K\"ahler manifold $(N,J,g,\omega)$ with an integer K\"ahler class $[\omega]$, we can form - via the Boothby-Wang construction - a regular Sasaki structure on the circle bundle over $N$ corresponding to the class $[\omega]$, so a $(f,n+2)$-extremal K\"ahler metric on $N$ corresponds to an extremal Sasaki structure which might be irregular. Moreover, the two examples above coincide when $n=2$, so for $p=4$, weighted extremal K\"ahler twins on a Hodge compact K\"ahler surface correspond to both Sasaki extremal metrics with the same underlying CR-structure and strongly Hermitian solutions of the Einstein-Maxwell equation.\\

In the K\"ahler setting, the study of twins have their origin in LeBrun's study of the Einstein-Maxwell equations \cite{Leb10,Leb15,Leb16}, where a curious twinning phenomenon appear for certain K\"ahler classes on the first Hirzebruch surface. In the Sasaki setting, while it is well known \cite{BGS06} that the set of extremal Sasakian structures is open when one can deform within the isotopy classes, the existence of twins which occurs when the CR structures are fixed appears, with some special case (Bochner flat \cite{Bry01,ApCaGa06}) exceptions, to be a discrete phenomenon and it is precisely this that we wish to study.

\begin{remark}
  More generally, given a compact torus $T$ and a $T$-Hamiltonian K\"ahler manifolds $(N, \omega,J,g)$ with a momentum map $\mu: M\ra P$ and moment polytope $P$, one can extend the notion above to more general {\it weighted K\"ahler twins} as a pair of triplets $(\omega, v,w)$, $(\omega, \tilde{v},\tilde{w})$ of respectively $(v,w)$-extremal K\"ahler metrics and $(\tilde{v},\tilde{w})$-extremal K\"ahler metrics in the sense of Lahdili \cite{Lah19} where, here, $v,w,\tilde{v},\tilde{w}$ are smooth functions on $P$. For example, any $v$-soliton admits a twin \cite{ApJuLa21}. To that extent, the notion seems very flexible while the occurrence of $p$-weighted extremal K\"ahler twins appears to be scarce.  
\end{remark}

We first explore the notion of twins on {\it Hirzebruch surfaces} which are compact, complex surfaces given as a projective bundle over $\bbc\bbp^1$ of the form $\bbf_{\hirz} =\bbp(\calo\oplus\calo(\hirz))$ for $\hirz\in \bbn$. As we recall in \S\ref{hirze} below, up to a dilatation, K\"ahler classes of the $\hirz$-th Hirzebruch surface $\bbf_{\hirz}$ are conveniently parameterized by $x\in (0,1)$ and each of these classes contains an important family of explicit K\"ahler metrics called {\it admissible}. In each class, one of these admissible metrics is the extremal K\"ahler metric originally build by Calabi \cite{cal82}.  

\begin{theorem}
Every Hirzebruch surface has an open, non-empty sub-cone of the K\"ahler cone such that every K\"ahler class in this sub-cone
admits a non-trivial pair of weighted extremal K\"ahler twins $(g, f_1)$ and $(g, f_2)$ of weight 4. For the trivial, first, and second Hirzebruch surfaces, this sub-cone is the entire K\"ahler cone.

Every Hirzebruch surface has a K\"ahler metric with integer primitive K\"ahler class, such that the Boothby-Wang constructed Sasaki manifold above this K\"ahler manifold contains a pair of non-trivial extremal Sasaki twins. 

Moreover, each admissible extremal metric on a Hirzebruch surface admits at most one extremal twin and the Sasaki-Einstein structure on the circle bundle corresponding to the canonical bundle of a Hirzebruch surface admits no twin.
\end{theorem}
The last statement is proved using the classification of extremal K\"ahler surfaces with toric Hamiltonian $2$-forms \cite{ApCaGa06, ACGT04, Leg09}. This toric formalism is also used here in \S \ref{ssSIMPLEX}, to prove by a direct and simple calculation that for any positive Killing potential $f$ of $(\bbc\bbp^n, g,J,\omega)$ with its Fubini-Study K\"ahler metric $g$, $(g,f)$ and $(g,1)$ are $(n+2)$-weighted extremal twins. This results can be extracted also from the results in \cite{Bry01,ApCaGa06} since the CR-structure of the $(2n+1)$-dimensional round sphere governs the Bochner-flatness of the (local) K\"ahler quotient which in turns ensures extremality in the Calabi sense. All that to say that the Fubini-Study K\"ahler metric of $\bbp^n$ has a $(n+1)$-dimensional cone of $(n+2)$-weighted extremal twins, but in all examples we know this is the only Sasaki-Einstein structure admitting a twin. In fact, identifying Killing potentials with an affine function on the moment polytope, we exhibit a combinatorial condition obstructing the existence of twins to a Sasaki-Einstein metric, see Lemma~\ref{lem:CombinCdt}. We use this criterion to prove the following, see Example~\ref{exBl3cp2}.

\begin{corollary} The standard regular Sasaki–Einstein structure on the circle bundle over $Bl_3(\bbc\bbp^2)$ the blow-up of $\bbc\bbp^2$ at the $3$ generic fixed points admits no extremal Sasaki twin. 
\end{corollary} 

We have a special interest in constant scalar Sasaki (cscS) structures which is a particular case of extremal Sasaki metrics that comes in rays into their Sasaki cones. Since the Tanaka-Webster scalar curvature, $Scal(\xi_f)$, of $\xi_f$ equals $Scal_{f,n+2}(g)/f$,
this extremal Sasaki structure is a cscS precisely when 
\begin{equation}\label{cscS1}
Scal_{f,n+2}(g)/f=\text{constant}.
\end{equation} The condition that both $(M, \cald, J,\xi_1)$ and $(M, \cald, J,\xi_f)$ are cscS gives a strong equation on $f$ via \eqref{cscS1}. Exploiting this, we are able to show the following result which implies that the (projective) set of solution is quadratic.  

\begin{theorem} Let $(M,\cald,J)$ be a pseudoconvex compact connected manifold on which acts a compact torus $T$ and denote $\kt^+$ the Sasaki cone. Denote by $\mathrm{cscS}(\kt^+) \subset \kt^+$ the set of rays of cscS structure (i.e for each $\xi \in \mathrm{cscS}(\kt^+)$, $(M,\cald,J,\xi)$ is cscS). Then, 
\begin{itemize}
    \item[1)] $\mathrm{cscS}(\kt^+)/\bbr_{>0} \subset \kt^+/\bbr_{>0}$ meets every line in $\kt^+/\bbr_{>0}$ in at most two points;    
    \item[2)] $\mathrm{cscS}(\kt^+)$ lies in the boundary of its convex hull. 
\end{itemize}
\end{theorem}

In the examples we studied we did not find more than two cscS twins on the same CR structure. In particular, there is no more than two cscS twins on the CR-structure underlying a toric Sasaki manifold with moment quadrilaterals, see \S\ref{quadrilaterals}. Note that this confirms that among the three cscS structures of \cite{Leg10} only two of them are twins. Finally, we also point out that there is an abundance of these two cscS twins. See Example \ref{exemplecscStwins}.\\ 

In \S\ref{sasakibackground} we will give a brief review of Sasaki geometry and in particular describe the Boothby-Wang construction alluded to at various points in this paper.
In \S\ref{KahlerHirze} we will focus on the K\"ahler setting and treat the Hirzebruch surfaces in great detail and produce many examples 
of weighted extremal K\"ahler metrics with weight $p=4$. See Propositions \ref{page}, \ref{1st2ndexistence}, \ref{higherhirzeexistence}, as well as \S\ref{0hirze}. The weight $p=4$ is chosen since the complex dimension $n$ is two and so here $p$ is both 
\begin{itemize}
\item equal to $2n$,
allowing us to generalize the twinning phenomenon appearing among the strongly Hermitian solutions to the Einstein-Maxwell equations on the first Hirzebruch surface, and
\item equal to $n+2$, allowing us, when appropriate, to draw conclusion pertaining to extremal Sasaki twins.
\end{itemize}
In \S\ref{highergenus} we extrapolate the work from \S\ref{KahlerHirze} to show that we also have examples of weighted extremal K\"ahler metrics with weight $p=4$ on minimal ruled surfaces over higher genus Riemann surfaces.
This is to show that the twinning phenomenon is not confined to toric structures.
\S~\ref{sastwins} is dedicated to extremal Sasaki twins and the last section contains the local or transversal computations on the quadrilaterals.  


\begin{ack} We would like to thank Vestislav Apostolov for making us aware of the work by Futaki and Ono \cite{FutOno18,FutOno19} and Viza De Souza \cite{VizaDS21}. See Remark \ref{FODeSouza}. We also would like to thank Claude LeBrun and Roland P{\'u}\v{c}ek for their interest in our work.
\end{ack}

\section{Brief introduction to Sasaki Geometry and the Boothby-Wang Construction\label{sasakibackground}}

Sasaki geometry can be thought of as the odd dimensional sister geometry of K\"ahler geometry. Indeed, up to an infinitesimal deformation, a Sasakian structure is an $S^1$ orbibundle over a K\"ahler orbifold. This is the orbifold version of a construction due originally to Boothby and Wang \cite{BoWa58} in the context of contact and symplectic geometries. When the Reeb vector field of the contact structure has closed orbits the symplectic base is an orbifold, and when the symplectic base has a K\"ahlerian structure the total space of the $S^1$ orbibundle has a Sasakian structure. 
We refer to this construction as the orbifold Boothby-Wang construction, or just the Boothby-Wang construction for short when the context is clear. It is important to realize that there are Sasakian structures which do not have a well-defined base. These arise when the orbits of the Reeb vector field are not closed, but their closure generates a higher dimensional torus. To summarize, Sasaki geometry can be understood, up to infinitesimal deformations, in terms of the bundle structure

\centerline{circle $\lra$ Sasaki $\lra$ K\"ahler.}
\noindent
For more details see Section 7.1 in \cite{BG05}.

In this paper, whenever we are considering a Sasaki structure $(\cald, J, \xi)$ with Sasaki metric $g$ we will use $Scal(g)$ to denote the Tanaka-Webster scalar curvature $Scal(\xi)$ of $\xi$. This also corresponds to the transverse scalar curvature, i.e., the scalar curvature of the transverse (local) K\"ahler structure. [The scalar curvature of the Sasaki metric $g$ (as a Riemannian metric on the underlying $2n+1$-dimensional manifold) is equal to $Scal(g)-2n$, and hence just differs from $Scal(g)$ by a constant.]

\begin{remark}
Given the Boothby-Wang construction it is occasionally useful to mix the jargon between Definitions \ref{twinsKahler} and \ref{twinsSasaki}. Suppose $g$ is an extremal K\"ahler metric such that it is the transverse K\"ahler structure of an extremal Sasaki structure $(\cald,J,\xi)$. Further, suppose  that $(g,1)$ and $(g,f)$ are $p$-weighted extremal K\"ahler twins with $p=n+2$, i.e., $f>0$ is a non-constant Killing potential and $g$ is $(f,n+2)$-extremal. Then we know from the discussion in \S\ref{intro} that $(g,f)$ corresponds to an extremal Sasaki structure $(\cald,J,\xi_f)$ and therefore $(\cald,J,\xi)$ and $(\cald,J,\xi_f)$ are extremal Sasaki twins. In this case, and when it is convenient, we might say that the K\"ahler metric $g$ ``has a extremal Sasaki twin.''
\end{remark}

\section{K\"ahler setting: $4$-Weighted Extremal K\"ahler twins on Hirzebruch Surfaces}
\label{KahlerHirze}
We will now explore a particularly attractive and well-known family of compact, complex surfaces, namely the \emph{Hirzebruch surfaces}. These are total spaces of $\bbc\bbp^1$-bundles over $\bbc\bbp^1$ that can be classified by non-negative integers with the integer zero corresponding to the trivial Hirzebruch surface $\bbf_0=\bbc\bbp^1\times \bbc\bbp^1$. In general we have the $\hirz^{th}$ Hirzebruch surface, $\bbf_\hirz$, for each positive integer $\hirz$.

Let $\hirz\in {\mathbb Z}^+$ and let $\bbf_\hirz$ denote the $\hirz$th Hirzebruch surface of the form
${\mathbb P}({\mathcal O} \oplus {\mathcal O}(\hirz)) \rightarrow {\mathbb C}{\mathbb P}^1$,
where ${\mathcal O}(\hirz)$ is the
holomorphic line bundle of degree $\hirz$ over ${\mathbb C}{\mathbb P}^1$, and
${\mathcal O}$ is the trivial holomorphic line bundle. 

We know that $\bbf_\hirz$ admits no cscK metrics when $\hirz>0$ 
(see \cite{cal82}).
Let $g_{{\mathbb C}{\mathbb P}^1}$ be the
K\"ahler metric
on ${\mathbb C}{\mathbb P}^1$ of constant scalar curvature $2s$, with K\"ahler form
$\omega_{{\mathbb C}{\mathbb P}^1}$, such that
$c_{1}({\mathcal O}(\hirz)) = [\frac{\omega_{{\mathbb C}{\mathbb P}^1}}{2 \pi}]$.
Let ${\mathcal K}_{{\mathbb C}{\mathbb P}^1}$ denote the canonical bundle of ${\mathbb C}{\mathbb P}^1$. Since $c_1({\mathcal K}_{{\mathbb C}{\mathbb P}^1}^{-1}) = [\rho_{{\mathbb C}{\mathbb P}^1}/2\pi]$, where 
$\rho_{{\mathbb C}{\mathbb P}^1}$ denotes the Ricci form, we must have that $s= 2/\hirz$. 

The natural fiber wise $\mathbb{C}^*$-action on ${\mathcal O}(\hirz) \rightarrow {\mathbb C}{\mathbb P}^1$
extends to a holomorphic
$\mathbb{C}^*$-action on $\bbf_\hirz$. The open and dense set $\bbf_\hirz^0$ of stable points with
respect to the
latter action has the structure of a principal $\mathbb{C}^*$-bundle over the stable
quotient.
A choice of a Hermitian norm on the fibers of ${\mathcal O}(\hirz)$, whose Chern connection has curvature equal to $\omega_{{\mathbb C}{\mathbb P}^1}$, induces via a Legendre transform a function
$\mathfrak{z}:\bbf_\hirz^0\rightarrow (-1,1)$ whose extension to $\bbf_\hirz$ consists of the critical manifolds
$E_{0}:=\mathfrak{z}^{-1}(1)=P({\mathcal O} \oplus 0)$ and
$E_{\infty}:= \mathfrak{z}^{-1}(-1)=P(0 \oplus {\mathcal O}(\hirz))$.
These  are respectively the zero and infinity section of $\bbf_\hirz \rightarrow
{\mathbb C}{\mathbb P}^1$. It is well-known that  $E_{0}$ and $E_{\infty}$ have the property that $E_{0}^{2} = \hirz$ and $E_{\infty}^{2} = -\hirz$,
respectively. If $C$ denotes a fiber of the ruling $\bbf_\hirz \rightarrow
{\mathbb C}{\mathbb P}^1$, then $C^{2}=0$, while $C \cdot E_{i} =1$ for both, $i=0$ and
$i=\infty$. Any real cohomology class in the two dimensional space
$H^{2}(\bbf_\hirz, {\mathbb R})$ may be written as a linear combination of
the Poincar\'e duals of $E_{0}$ and $C$, i.e. $(m_{1} E_{0} +m_{2} C)^*$.
Thus, we may think of $H^{2}(\bbf_\hirz, {\mathbb R})$ as ${\mathbb R}^2$,
with coordinates $(m_{1},m_{2})$. So the K\"ahler cone ${\mathcal K}$ may
be identified with
${\mathbb R}_{+}^2=\{ (m_{1}, m_{2})\,  | \; m_{1} > 0,
m_{2} > 0 \}$.

\subsection{Admissible metrics on Hirzebruch Surfaces}\label{hirze}

The term \emph{admissible K\"ahler metrics} were introduced as such in \cite{ACGT08}, but on a Hirzebruch surface they date much further back to \cite{cal82}. What follows is a quick overview on how to build such metrics
on $\bbf_\hirz$. We will use the notation from \cite{ACGT08} to which we refer for references as well as more technical details on what follows below.

Let  $\theta$ be a connection one form for the
Hermitian metric on $\bbf_\hirz^0$, with curvature
$d\theta = \omega_{{\mathbb C}{\mathbb P}^1}$. Let $\Theta$ be a real function which is smooth and positive on the open interval
$(-1,1)$ and continuous on the closed interval $[-1,1]$. Let $x$ be a real number such that $0 < x < 1$.
Then an admissible K\"ahler metric
is given on $\bbf_\hirz^0$ by
\begin{equation}\label{metric}
g_x  =  \frac{1+x \mathfrak{z}}{x} g_{{\mathbb C}{\mathbb P}^1}
+\frac {d\mathfrak{z}^2}
{\Theta (\mathfrak{z})}+\Theta (\mathfrak{z})\theta^2\,
\end{equation}
with K\"ahler form
\begin{equation*}
\omega_x =  \frac{1+x \mathfrak{z}}{x}\omega_{{\mathbb C}{\mathbb P}^1}
+d\mathfrak{z}\wedge \theta\,. \label{kf}
\end{equation*}
The complex structure yielding this
K\"ahler structure is given by the pullback of the base complex structure
along with the requirement
\begin{equation}\label{complex}
Jd\mathfrak{z} = \Theta \theta.
\end{equation}
 The function $\mathfrak{z}$ is
Hamiltonian
with $K= J\grad \mathfrak{z}$ a Killing vector field. Observe that $K$
generates the circle action which induces the holomorphic
$\mathbb{C}^*$- action on $\bbf_\hirz$ as introduced above.
In fact, $\mathfrak{z}$ is the moment
map on $\bbf_\hirz$ for the circle action, decomposing $\bbf_\hirz$ into
the free orbits $\bbf_\hirz^0 = \mathfrak{z}^{-1}((-1,1))$ and the special orbits
$\mathfrak{z}^{-1}(\pm 1)$. Finally, $\theta$ satisfies
$\theta(K)=1$.
In order that $g_x$ (be a genuine metric and) extend to all of $\bbf_\hirz$,
$\Theta$ must satisfy the positivity and boundary
conditions
\begin{align}
\label{positivity}
(i)\ \Theta(\mathfrak{z}) > 0, \quad -1 < \mathfrak{z} <1,\quad
(ii)\ \Theta(\pm 1) = 0,\quad
(iii)\ \Theta'(\pm 1) = \mp 2.
\end{align}
It is convenient to define a function $F(\mathfrak{z})$ by the formula
\begin{equation}\label{theta}
\Theta(\mathfrak{z})= \frac{F(\mathfrak{z})}{(1+x
\mathfrak{z})}.
\end{equation}
Since $(1+x
\mathfrak{z})$ is positive for $-1<\mathfrak{z}<1$, conditions
\eqref{positivity}
imply the following equivalent conditions on $F(\mathfrak{z})$:
\begin{align}
\label{positivityF}
(i)\ F(\mathfrak{z}) > 0, \quad -1 < \mathfrak{z} <1,\quad
(ii)\ F(\pm 1) = 0,\quad
(iii)\ F'(\pm 1) = \mp 2(1 \pm x).
\end{align}

The construction of admissible K\"ahler metrics is based on the symplectic viewpoint. 
Different choices of $F$ yield different complex structures that are all compatible with the same fixed symplectic form $\omega_x$. However, for each $F$ there is an $S^1$-equivariant diffeomorphism pulling back $J$ to
 the original fixed complex structure on $\bbf_\hirz$ in such a way that the K\"ahler form of the new K\"ahler metric is in the same cohomology class as $\omega_x$ \cite{ACGT08}. 
 Therefore, with all else fixed, we may view the set of the functions $F$ satisfying \eqref{positivityF} as parametrizing a certain family of K\"ahler metrics within the same K\"ahler class of $\bbf_\hirz$.

One easily checks that the K\"ahler class of an admissible metric \eqref{metric} is given by
\begin{equation}\label{class1}
[\omega_x] = 4\pi E_{0}^*+ \frac{2\pi(1-x)\hirz}{x} C^*.
\end{equation}
Up to an overall rescale, every K\"ahler class in the K\"ahler cone may be represented by an admissible K\"ahler metric. Note that when $x\in \bbq\cap (0,1)$, $[\omega_x]$ can be rescaled to be an integer cohomology class. Thus
the $x\in \bbq\cap (0,1)$ cases are of particular interest from a Sasaki geometry point of view.

\subsection{Weighted Extremal Admissible K\"ahler metrics}\label{admisEMsoln}
As described in \S 3 of \cite{BHLT23}, by employing Proposition 11 of \cite{ACGT08} and (the second half of) Theorem 1 in \cite{ApMaTF18}, we have that (up to scale) every K\"ahler class on $\bbf_\hirz$ admits an admissible $(c\,\gz+1,4)$-extremal metric for any choice of $c\in (-1,1)$.

Following \cite{ApMaTF18} (with some adjustment in notation), we have that for a given $c\in (-1,1)$ and $x\in (0,1)$, the
admissible $(c\,\gz+1,4)$-extremal metric $g_{x,c}$ is determined by the function
\begin{equation}\label{wextrpol5mnf}
F_{x,c}(\mathfrak{z}):=(c\gz+1)^3 \left[ \frac{2(1-x)}{(1-c)^3}(\gz+1) + \int_{-1}^\gz Q(t)(\gz-t)\,dt\right],
\end{equation}
where 
$$Q(t) = \frac{2xs}{(ct+1)^3} - \frac{(A_1 t+A_2)(1+xt)}{(ct+1)^5}$$
and $A_1$ and $A_2$ are the unique solutions to the linear system
\begin{equation}\label{wextrpol5mnf2}
\begin{array}{ccl}
\alpha_{1,-5}A_1+\alpha_{0,-5}A_2&=& 2\beta_{0,-3}\\
\\
\alpha_{2,-5}A_1+\alpha_{1,-5}A_2&=& 2\beta_{1,-3}
\end{array}
\end{equation}
with
$$
\begin{array}{ccl}
\alpha_{r,-k} & = & \int_{-1}^1(ct+1)^{-k}t^r(1+xt)\,dt\\
\\
\beta_{r,-k} & = & x s\int_{-1}^1(ct+1)^{-k}t^r\,dt \\
\\
&+ & (-1)^r(1-c)^{-k}(1-x) +(1+c)^{-k}(1+x).
\end{array}
$$
It is straightforward to calculate that
$$
F_{x,c}(\mathfrak{z})=(1-\gz^2)P_{x,c}
$$
where 
\begin{equation}\label{polP}
\begin{array}{ll}
P_{x,c}(\gz) =  \frac{c^2 s x+3 c^2 x^2-c^2-2 c s x^2+3 c x^3-7 c x+s x^3-4 x^2+6}{2 \left(3 c^2 x^2-c^2-4 c x-x^2+3\right)}
+x\, \gz +
&\frac{(c-x) \left(-c s x+3 c x^2-c+s x^2-2 x\right)}{2 \left(3 c^2 x^2-c^2-4 c x-x^2+3\right)}\,\gz^2.
\end{array}
\end{equation}
Note that $3 c^2 x^2-c^2-4 c x-x^2+3=(3x^2-1) c^2-4x c +3-x^2>0$ for all $x\in (0,1)$ and $c\in (-1,1)$.
Note also that for any $c\in (-1,1)$ and $x\in (0,1)$, all of the conditions in \eqref{positivityF} are satisfied here\footnote{If we extend things to higher genus geometrically ruled surfaces (as we will in Appendix \ref{highergenus}), (i) of  \eqref{positivityF}
would not be guaranteed in all cases. See e.g \cite{To-Fr98}, for the classic (extremal) case where $c=0$.}.
Then, $Scal_{c\,\gz+1,4}(g_x)$ is constant ($A_1=0$) iff 
\begin{equation}\label{em}
\alpha_{0,-5}\beta_{1,-3}-\alpha_{1,-5}\beta_{0,-3}=0,
\end{equation}
which is equivalent to
\begin{equation}\label{em2}
\left(c^2 x-2 c+x\right) \left(c^2 s x-2 c x-s x+2\right)=0.
\end{equation}
The solutions to Equation \eqref{em2} just recover the examples of the strongly Hermitian solutions of Einstein-Maxwell equations already produced by LeBrun in \cite{Leb16}.

Similarly, \eqref{cscS1} (with $f=c\gz+1$ and complex dimension $n=2$) is satisfied precisely when
\begin{equation}\label{cscS2}
\alpha_{0,-4}\beta_{1,-3}-\alpha_{1,-4}\beta_{0,-3}=0
\end{equation}
(see equation (10) in \cite{BHLT23}) and \eqref{cscS2} is equivalent to
\begin{equation}\label{cscS3}
(3x^2+sx-1)c^3 -x(6+sx)c^2  +(5-sx+x^2)c+ (sx-2)x=0.
\end{equation}
When $x\in \bbq\cap (0,1)$, this leads to csc Sasaki metrics on $S^3$-bundles over $S^2$. These metrics have been studied in e.g.
\cite{BoPa10, Leg10}.
Note that for a given value of $s=2/\hirz$ and $x\in (0,1)$, \eqref{cscS3} has a unique solution. This can be seen directly (following a similar discussion in \S4.1 of \cite{BTF23}) by making
a change of variable $c=\frac{1-\hat{c}}{1+\hat{c}}$ with $0<\hat{c}<+\infty$  so that \eqref{cscS3} rewrites to
$q(\hat{c})=0$, where $q(\hat{c})$ is the polynomial
$$q(\hat{c})=-(1+x)^2\hat{c}^3-(1+x)(2(1-x)-sx)\hat{c}^2+(1-x)(2+2x-sx)\hat{c}+(1-x)^2.$$
Since $s\leq 2$ and $0<x<1$, it is not hard to see that $q(\hat{c})$ has precisely one positive real root\footnote{This observation is pertaining to a proper subcone of the Sasaki cone and, in particular, does not in any way imply uniqueness of csc Sasaki metrics in the (entire) Sasaki cone. From \cite{Leg10} we know that such uniqueness does not hold in general.}. The uniqueness part uses 
Descarte's rule of signs. A simple example can be obtained by taking $s=2$ and $x=\frac{1}{3}$ in which case Equation \eqref{cscS3} becomes a quadric with the solution $c=\frac{5-2\sqrt5}{5}$.

\subsection{$4$-Weighted Extremal Admissible K\"ahler twins} \label{hirzetwins}

For a given $x\in (0,1)$, we now study cases where we have two distinct values $a,b\in (-1,1)$ so that
$F_{x,a}=F_{x,b}$. This means that the corresponding admissible K\"ahler metric, $g_x:=g_{x,a}=g_{x,b}$ is a $(f,4)$-extremal K\"ahler metric with respect to two different
Killing potentials, $f_1=a\gz+1$ and $f_2=b\gz +1$ which are then a pair of $4$-weighted extremal K\"ahler twins $(g,f_1)$ and $(g,f_2)$.\\

We already know that such cases exist. Indeed, examples can be found among LeBrun's strongly Hermitian solutions of Einstein-Maxwell equations on the first Hirzebruch surface \cite{Leb16} (see the discussion
in \S 3.2 of \cite{KoTF16}): For each $4/5<x<1$, there exist pairs of admissible $4$-weighted extremal twins $(g_x,a\gz+1)$ and $(g_x,b\gz+1)$ such that both $c=a$ and $c=b$ solve \eqref{em2}, that is $Scal_{a\gz+1,4}(g_x)$ and $Scal_{b\gz+1,4}(g_x)$ are both constant. 
Moving forward, we will refer to these as the {\em LeBrun Einstein-Maxwell Twins}.

Another example on the first Hirzebruch surface is the case of the Page metric, $h$, \cite{Pag} which is Hermitian Einstein, and its companion extremal K\"ahler metric $g$, with (up to scale)
$f_1=Scal(g)$ and $f_2=1$ (see \cite{Derd83}) and $h=f_1^{-2}g$. Here, $Scal_{f_1,4}=Scal(h)$ is constant, but $Scal_{f_2,4}=Scal(g)$ is not constant. We will call this pair the {\em Page Twins}. We encourage the reader to read \S 5 of \cite{Leb16} for a more detailed discussion of the Page metric in the first Hirzebruch surface.

What we are about to discover is that there is an {\em abundance} of such $4$-weighted extremal K\"ahler twins on $\bbf_\hirz$.\\ 

For a given $x\in(0,1)$ and $a\in (-1,1)$, consider the admissible $(a\gz+1,4)$-extremal K\"ahler metric with $P_{x,a}$ as in \eqref{polP}. 
A potential twin would be manifested by a value $b\in (-1,1)$, with $b\neq a$, such that
$P_{x,b}=P_{x,a}$, i.e.,
$$
\begin{array}{ccl}
 \frac{b^2 s x+3 b^2 x^2-b^2-2 b s x^2+3 b x^3-7 b x+s x^3-4 x^2+6}{2 \left(3 b^2 x^2-b^2-4 b x-x^2+3\right)} & = & \frac{a^2 s x+3 a^2 x^2-a^2-2 a s x^2+3 a x^3-7 a x+s x^3-4 x^2+6}{2 \left(3 a^2 x^2-a^2-4 a x-x^2+3\right)}\\
 \\
 \frac{(b-x) \left(-b s x+3 b x^2-b+s x^2-2 x\right)}{2 \left(3 b^2 x^2-b^2-4 b x-x^2+3\right)}&=&\frac{(a-x) \left(-a s x+3 a x^2-a+s x^2-2 x\right)}{2 \left(3 a^2 x^2-a^2-4 a x-x^2+3\right)}
 \end{array}
 $$
It is straightforward to calculate that, bearing in mind $x\in (0,1)$ and $a,b\in (-1,1)$, this system is equivalent to
one equation:
$$
(a-b)(x (1 - 2 s x + x^2)+a(1 + s x - 3 x^2 + s x^3)+b(1 + s x - 3 x^2 + s x^3+a x (-1 - 2 s x + 3 x^2)))=0.
$$
Since we are looking for solutions with $b\neq a$, this further
simplifies to 
\begin{equation}\label{twins}
x (1 - 2 s x + x^2)+(1 + s x - 3 x^2 + s x^3)a+(1 + s x - 3 x^2 + s x^3)b-x (1 + 2 s x - 3 x^2)ab=0.
\end{equation}
Solving for $b$ we get
\begin{equation}\label{whatbis}
b=-\frac{x (1 - 2 s x + x^2)+a(1 + s x - 3 x^2 + s x^3)}{(1 + s x - 3 x^2 + s x^3-a x (1 + 2 s x - 3 x^2))}.
\end{equation}
As long as $s=2/\hirz$, $x\in (0,1)$, and $a\in (-1,1)$ are such that \newline
$(1 + s x - 3 x^2 + s x^3-a x (1 + 2 s x - 3 x^2))\neq 0$ and moreover 
$b$ as given in \eqref{whatbis} satisfies
$-1<b<1$, we have a genuine solution
$b\in (-1,1)$, which will give us a pair of $4$-weighted extremal K\"ahler twins, $(g_x,a\gz+1)$ and $(g_x,b\gz+1)$, provided $b\neq a$.
If everything else is okay but $b=a$, it just means that the twins coincide with each other which would correspond to a bifurcation point.

\begin{remark}
By the nature of the equation \eqref{twins}, we see that (at least generically (meaning the corresponding conic section is not degenerate),
for a given
admissible $(a\gz +1,4)$-extremal K\"ahler metric $g_x$ there is at most one other value $b\in (-1,1)\setminus\{a\}$ such that $g_x$ is also $(b\gz +1,4)$-extremal.

Assuming the non-degeneracy of \eqref{twins}, we also note that if $x\in \bbq\cap (0,1)$ (giving us an integer K\"ahler class up to rescale) and if $(a,b)$ is a solution to \eqref{twins}, then $a\in \bbq \iff b \in \bbq$.
\end{remark}

\begin{example}[Generalizing the Page twins]
\label{pagetwins}
Let us assume $\hirz=1$, so $s=2$. That is, $\bbf_1$ is the first Hirzebruch surface.
Then \eqref{whatbis} is equivalent to
$$
b=-\frac{x (1 - 4x + x^2)+a(1 + 2 x - 3 x^2 + 2 x^3)}{(1 + 2 x - 3 x^2 + 2 x^3+a x (-1 - 4 x + 3 x^2))},
$$
and if we further assume $a=0$, we get
$$
b=-\frac{x (1 - 4x + x^2)}{(1 + 2 x - 3 x^2 + 2 x^3)}.
$$
Notice first that the denominator in the above expression is positive for all $x\in (0,1)$. Further, with $b$ as given above
$$b-(-1)=\frac{(x+1) \left(x^2+1\right)}{2 x^3-3 x^2+2 x+1}>0$$
and
$$1-b= \frac{(1-x) \left(1+4x-3 x^2\right)}{2 x^3-3 x^2+2 x+1}  >0$$
and hence
$$-1<-\frac{x (1 - 4x + x^2)}{(1 + 2 x - 3 x^2 + 2 x^3)}<1.$$
Finally notice that $b\neq 0$ as long as $x\in (0,1)\setminus\{2-\sqrt{3}\}$.
Thus in this case we have the following result that generalizes the Page twins. When $x=2-\sqrt{3}$, we get $b=0=a$ and the two twins are identical ($x=2-\sqrt{3}$ is a bifurcation point).
\end{example}

From this example we obtain the following result:

\begin{proposition}\label{page}
For all $x\in (0,1)\setminus\{2-\sqrt{3}\}$,
Calabi's (admissible) extremal K\"ahler metric, $g_x$, viewed as a $(1,4)$-extremal K\"ahler metric, is ``twinning'' with the
$(b\gz+1,4)$-extremal metric $g_x$, where $b\in (-1,0)\cup(0,1)$ is given by
$b=-\frac{x (1 - 4x + x^2)}{(1 + 2 x - 3 x^2 + 2 x^3)}$. In other words, $g_x$ is both extremal and $(b\gz+1,4)$-extremal and
we have a pair of $4$-weighted extremal K\"ahler twins $(g_x,1)$ and $(g_x,b\gz+1)$.
\end{proposition}

\begin{remark}
Note that with $c=b=-\frac{x (1 - 4x + x^2)}{(1 + 2 x - 3 x^2 + 2 x^3)}$ and $s=2$, \eqref{em} simplifies to
$$-3 + 18 x^2 - 16 x^3 + 5 x^4=0.$$
The solution, $x_P\in (0,1)$, of this equation has an explicit, albeit complicated, expression.
See e.g. \cite{Leb16} for the explicit K\"ahler class in question and hence a way to find the exact value of $x_P$.
Approximately we have $x_P\approx 0.52$. This corresponds to the K\"ahler class
containing the extremal K\"ahler metric conformal (via $(b\gz+1)^{-2}$) to the Page metric which as previously mentioned is Einstein. 
Thus for $x=x_P\approx 0.52$, Proposition \ref{page} gives us the Page twins.
\end{remark}

\begin{remark}
If we assume $x\in \bbq\cap (0,1) $ in Proposition \ref{page}, then after an appropriate rescale we can consider the Boothby-Wang constructed Sasaki manifold over
the extremal K\"ahler metric corresponding to (as rescale) of $g_x$. Then the $4$-weighted extremal K\"ahler twins represent two extremal rays in the Sasaki cone, where the two extremal Sasaki metrics are genuinely in the exact same 
Sasaki cone, that is, there is no deformation within the isotopy classes, or equivalently there is no movement within the corresponding K\"ahler classes. That is, the two metrics are extremal Sasaki twins.
\end{remark}

\subsection{Focus on the first two Hirzebruch surfaces, $\bbf_1$ and $\bbf_2$}\label{S1andS2}
Here we will assume $s=2$ or $s=1$. That is, $\hirz=1$ or $\hirz=2$ and $\bbf_\hirz$ is the first or second Hirzebruch surface. Note that equation \eqref{twins} define a conic section
$$Aa^2+Bab + Cb^2+Da+Eb+F=0$$
in the variables $a,b$, where $A=C=0$, $B=x (3 x^2-2s x - 1)$, $D=E=(1 + s x - 3 x^2 + s x^3)$, and $F=x (1 - 2s x + x^2)$.
Since $B\neq 0$ for $0<x<1$, we see that we get a hyperbola. In matrix form, the equation may be written
$$ \begin{pmatrix} a&b&1\end{pmatrix}\begin{pmatrix}A&B/2&D/2\\ B/2&C&E/2\\D/2&E/2&F \end{pmatrix}
\begin{pmatrix} a\\b\\1 \end{pmatrix}=0$$
and here
$$\begin{pmatrix}A&B/2&D/2\\ B/2&C&E/2\\D/2&E/2&F \end{pmatrix}=\begin{pmatrix}0&B/2&D/2\\ B/2&0&D/2\\D/2&D/2&F \end{pmatrix}.$$
The determinant of this matrix equals $$\frac{B\left(D^2-BF\right)}{4}=\frac{x (3 x^2-2s x - 1)(1-x^2)^2(1+2sx+(s^2-3)x^2)}{4}\underbrace{\neq}_{0<x<1, s=1,2} 0$$
and thus, for $s=2$ or $s=1$, the solution set to equation \eqref{twins} forms a non-degenerate hyperbola $H_x$ for all $x\in (0,1)$. 

Let $U$ denote the complement of the diagonal in the subset \newline
$\{(a,b)\,|\,0<|a|,|b|<1\}$ of the $ab$-plane.

\noindent
{\bf Claim:} $\forall x\in (0,1),\, H_x\cap U \neq \emptyset$.

\noindent
{\bf Proof of Claim:} The left-hand-side of \eqref{twins} equals a smooth function
$g(s,x,a,b)$. Now observe that, since $0<x<1$,
\begin{itemize}
\item $g(2,x,-1,0)=-(1 + x) (1 + x^2)<0$
\item $g(1,x,-1,0)=-(1-x^2)<0$
\item $g(2,x,0,1)=(1 - x) (1 + 4 x - 3 x^2)>0$
\item $g(1,x,0,1)=(1 - x) (1 + 3 x - 2 x^2)>0.$
\end{itemize}
This tells us that for both $s=2$ and $s=1$ and any $x\in (0,1)$, there must exist at least one solution to \eqref{twins} on the (open) line segment between $(-1,0)$ and $(0,1)$ in the $ab$-plane.
 Since this (open) line segment is inside the square $(-1,1)\times(-1,1)$, does not intersect the diagonal and does not intersect the lines $a=0$, $b=0$, we are done. \qed
 
Due to this claim we have the following result:
\begin{proposition}\label{1st2ndexistence}
In every K\"ahler class on the first and second Hirzebruch surface we have at least one genuine pair of $4$-weighted extremal K\"ahler twins $(g_x,a\gz+1)$ and $(g_x,b\gz+1)$ where neither of them are extremal.
\end{proposition}

Of course, since $H_x$ is a non-degenerate hyperbola the claim above actually ensures that we have an uncountable (one parameter) number of pairs of $4$-weighted extremal K\"ahler twins in each K\"ahler class.
To illustrate this, we include the figure below showing $H_{1/6}, H_{2/6}, H_{3/6}, H_{4/6}$, and $H_{5/6}$ for $s=2$. Note that here the curves move toward the top right corner as $x$ increases.

\bigskip

\includegraphics[scale=0.5]{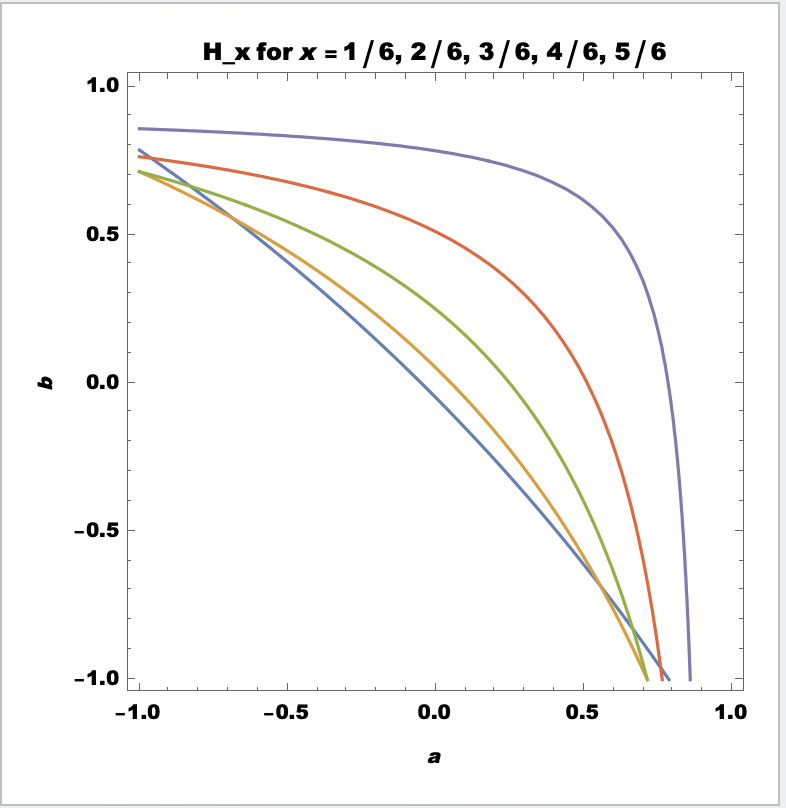}

\subsubsection{The LeBrun Einstein-Maxwell Twins on the First Hirzebruch Surface}
Let us now focus further on the first Hirzebruch surface. Since $s=2$, equation \eqref{twins} for potential $4$-weighted extrema K\"ahler twins $(g_x,a\gz+1)$ and $(g_x,b\gz+1)$
becomes

\begin{equation}\label{twins1stHirzebruch}
x (1 - 4 x + x^2)+(1 + 2 x - 3 x^2 + 2 x^3)a+(1 + 2 x - 3 x^2 + 2 x^3)b-x (1 + 4 x - 3 x^2)ab=0.
\end{equation}

Viewed another way, and as we saw in Example \ref{pagetwins}, for a given
admissible $4$-weighted extremal K\"ahler metric $(g_x,a\gz+1)$ (with $-1<a<1$) in the K\"ahler class given by $x\in (0,1)$, we have a potential twin $(g_x,b\gz+1)$ given by
\begin{equation}\label{potentialtwin}
b=-\frac{x (1 - 4x + x^2)+a(1 + 2 x - 3 x^2 + 2 x^3)}{(1 + 2 x - 3 x^2 + 2 x^3+a x (-1 - 4 x + 3 x^2))},
\end{equation}
as long as $-1<b<1$ (and - for a genuine twin - $b\neq a$).

Let us assume that $a\in (-1,1)$ solves \eqref{em2}, with $s=2$, and thus corresponds to one of the strongly Hermitian solutions of the Einstein-Maxwell equations\footnote{$h=(a\,\gz+1)^{-2}g_x$ is
{\em conformally K\"ahler, Einstein-Maxwell.}} due to LeBrun in \cite{Leb16}.
That is, $Scal_{a\,\gz+1,4}(g_x)$ is constant and 

\begin{equation}\label{em3}
\left(xa^2 -2 a+x\right) \left(  xa^2-x a +(1-x)\right)=0.
\end{equation}

As discovered in \cite{Leb16}, \eqref{em3} (as an equation in $a$) has solutions
$$
\begin{array}{ccl}
a_1&=&\frac{1-\sqrt{1-x^2}}{x}\quad\text{valid for all $0<x<1$}\\
\\
a_{21} &=& \frac{x-\sqrt{x(5 x-4)}}{2 x}\quad\text{valid for $4/5\leq x<1$}\\
\\
a_{22} &=&  \frac{x+\sqrt{x(5 x-4)}}{2 x} \quad\text{valid for $4/5\leq x<1$.}
\end{array}
$$
Note that when $x=4/5$, $a_1=a_{21}=a_{22}$, and for $4/5<x<1$, the three solutions are distinct. 
Thus, when $0<x\leq4/5$ we have precisely one solution and when $4/5<x<1$ we have three distinct solutions.

It is not hard to check that \eqref{twins1stHirzebruch} is solved by $(a,b)=(a_{21},a_{22})$ and this recovers precisely the
LeBrun Einstein-Maxwell Twins. If we denote the shared K\"ahler metric in the K\"ahler class given by $x$, corresponding to $a_{21}$ and $a_{22}$,
by $g_{x,2}$, then for each $x\in (4/5,1)$, the LeBrun Einstein-Maxwell twins is the pair $(g_{x,2},a_{21}\gz+1), (g_{x,2},a_{22}\gz+1)$.
On the other hand, if we denote the K\"ahler metric in the K\"ahler class given by $x$, corresponding to $a_{1}$, by $g_{x,1}$, we have
for all $x\in (0,1)$ a solution to the Einstein-Maxwell equations given by $(g_{x,1},a_1\gz+1)$. This solution also has a twin, $(g_{x,1},b_1\gz+1)$
obtained by letting $b_1=b$ when $a=a_1$ in \eqref{potentialtwin}\footnote{One can check that
$b_1=\frac{a_1 \left(a_1^4+4 a_1^3-6 a_1^2+12 a_1-3\right)}{a_1^4+12 a_1^3-10 a_1^2+4 a_1+1},$
and for $0<x<1$, $0<a_1<1$, we have $-1<b_1<1$.}

However, $(g_{x,1},b_1\gz+1)$ does not correspond to a solution of  \eqref{em3} and
$(g_{x,1},b_1\gz+1)$ is therefore not a solution to the Einstein-Maxwell equations. The figure below shows $a_1,b_1,a_{21}$, and $a_{22}$ as functions of
$x\in (0,1)$. The red and green curves correspond to the LeBrun Einstein-Maxwell Twins ($a_{21}$ and $a_{22}$), the blue curve is $a_1$ and the orange curve is $b_1$.

\includegraphics[scale=0.3]{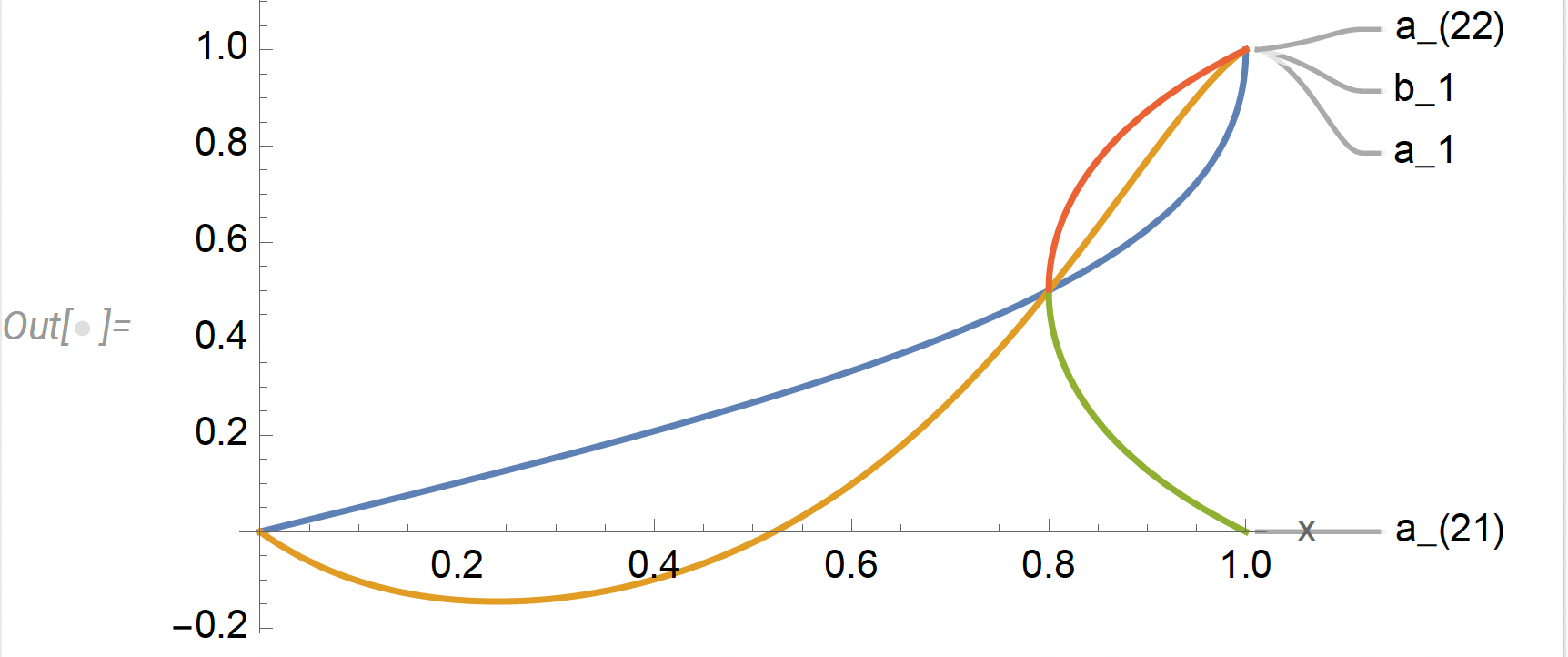}

At the $x$ value where the orange curve crosses the $x$-axis we have, once more,  the Page twins sitting on the orange and blue curve, respectively.

\begin{remark}\label{FODeSouza}
We would be remiss if we did not mention that, in addition to LeBrun's solutions in \cite{Leb16}, certain Hirzebruch surfaces also admit solutions to the strongly Hermitian Einstein-Maxwell equations where the corresponding $4$-weighted extremal K\"ahler metric is non-admissible, i.e., not of Calabi type, and ambitoric in the sense of \cite{ApCaGa16}. This existence result is due to work by Viza De Souza \cite{VizaDS21} building on work by Futaki and Ono \cite{FutOno18,FutOno19} as well as Apostolov and Maschler \cite{ApMa19}. It would be very interesting to explore whether any of these solutions could be part of a pair of $4$-weighted extremal K\"ahler twins.
\end{remark}

\subsubsection{The Sasaki Einstein metric over the first Hirzebruch surface and related twins}
Continuing our focus on the first Hirzebruch class we note that for $x=1/2$, the admissible K\"ahler class agrees (up to scale) with the first Chern class.
Thus the corresponding Boothby-Wang constructed Sasaki structure is Gorenstein and the unique solution to \eqref{cscS1} (with complex dimension $n=2$ and manifested in \eqref{cscS3}), corresponds to a
Sasaki-Einstein metric (on an $\eta$-Einstein ray). It is not hard to check that the solution to \eqref{cscS3} for $s=2$ and $x=1/2$ is $c=\frac{1}{3} \left(4-\sqrt{13}\right)$.
Setting $s=2$, $x=1/2$ and $a=\frac{1}{3} \left(4-\sqrt{13}\right)$ in \eqref{twins}, yields the trivial solution
$b=\frac{1}{3} \left(4-\sqrt{13}\right)=a$. Thus, the Sasaki Einstein metric does not have an extremal Sasaki twin in this sub cone (given by $f=c\gz+1$) of its Sasaki cone. [We shall see later in \S \ref{quadrilaterals} that it has no twin at all.]
On the other hand for the cscS solution with $s=2$ and $x=\frac{1}{3}$ we obtain a weighted extremal Sasaki  twin by taking 
$$a=c=\frac{5-2\sqrt{5}}{5}, \qquad b=\frac{-45 + 19 \sqrt{5}}{25 + 9 \sqrt{5}}.$$

\subsection{The third and higher Hirzebruch surfaces}

Focussing on the Hirzebruch surfaces of degree three or higher, we now
assume that $s=2/\hirz$ with $\hirz=3,4,5,\dots$. This means that $s \leq 2/3$ and 
one can check that then the hyperbola in $a,b$ given by \eqref{twins} will be degenerate for exactly two values of $x$ in $(0,1)$, namely when 
$x=\frac{1}{3} \left(\sqrt{s^2+3}+s\right)$ or $x=-\frac{s+\sqrt{3}}{s^2-3}$ (See the discussion at the beginning of \S \ref{S1andS2}). 
It can also be checked directly that
for $x=\frac{1}{3} \left(\sqrt{s^2+3}+s\right)$ or $x=-\frac{s+\sqrt{3}}{s^2-3}$, \eqref{twins} has no solutions with $-1<a,b<1$.

Note now that since $0<s\leq 2/3$ we have that
$0<s<\frac{1}{3} \left(\sqrt{s^2+3}+s\right)<-\frac{s+\sqrt{3}}{s^2-3}<1.$
Thus, with $s\leq 2/3$, the solution set to equation \eqref{twins} forms a non-degenerate hyperbola $H_x$ for all $x\in (0,s]\subset (0,1)$. 
Let $W$ denote the complement of the diagonal in the subset
$(-1,1)\times(-1,1)$ of the $ab$-plane.

\bigskip

\noindent
{\bf Claim:} $\forall s=2/\hirz \leq 2/3, \forall x\in (0,s],\, H_x\cap W \neq \emptyset$.

\bigskip

\noindent
{\bf Proof of Claim:} 
Assuming $0<x\leq s\leq 2/3$, it is not hard to see that, if we view the left hand side of equation \eqref{twins} as a smooth function $g(s,x,a,b)$, then,
$$
\begin{array}{ccl}
g(s,x,-1,-1)&=&-2(1+x)^2(1+(s-2)x) <0\\
\\
g(s,x,1,0)&=& (1-x)(1+(2-x)x+sx(1-x))>0.
\end{array}
$$
This tells us that for $0<x\leq s=2/\hirz \leq 2/3$, there is at least one solution to \eqref{twins} on the open line segment between $(-1,-1)$ and $(1,0)$ in the $ab$-plane. 
Since this open line segment is inside the square $(-1,1)\times(-1,1)$ and does not intersect the diagonal, we are done. \qed

Due to this claim we have the following result:
\begin{proposition}\label{higherhirzeexistence}
For each K\"ahler class determined by $0<x\leq 2/\hirz$ on the third and higher Hirzebruch surfaces ($\bbf_\hirz$ for $\hirz=3,4,5,\dots$) 
we have at least one genuine pair of $4$-weighted extremal K\"ahler twins $(g_x,a\gz+1)$ and $(g_x,b\gz+1)$.
\end{proposition}

Once again, due to non-degeneracy of the hyperbola, for each K\"ahler class determined by $0<x\leq 2/\hirz$, 
we have an uncountable (one parameter) number of pairs of $4$-weighted extremal K\"ahler twins. The figure below shows $H_x$ for $x=1/\hirz$ and
$s=2/\hirz$ when $\hirz=3, 4, 5, 6, 7, 8$. Here the curves are moving towards the upper right corner as $\hirz$ increases.

\bigskip

\includegraphics[scale=0.4]{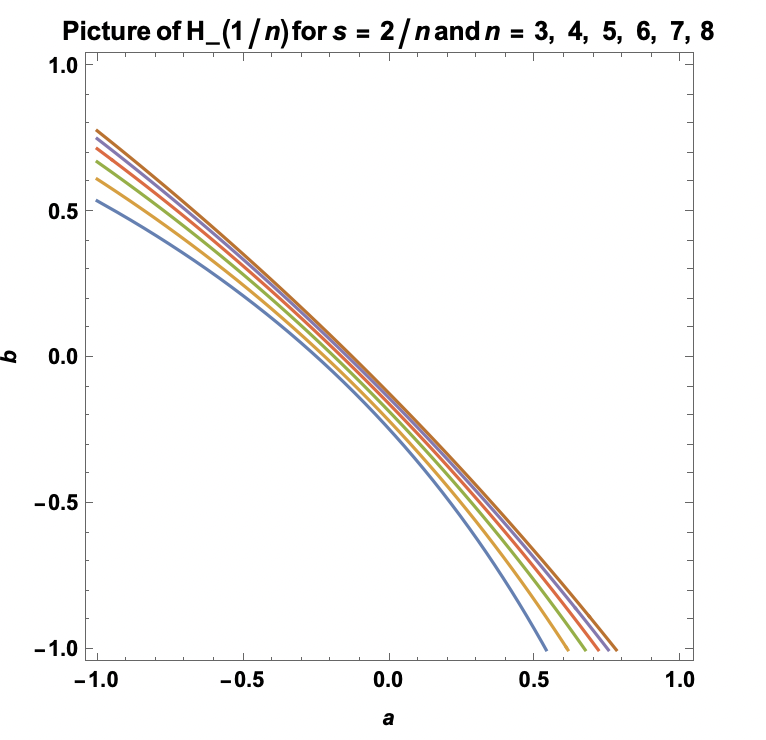}

\bigskip

Finally, we also present some sporadic explicit solutions:
\begin{itemize}
\item Equation \eqref{twins} is solved for $s=2/\hirz, x=1/\hirz, a=1/\hirz,$ and $b=-2/\hirz$ and for $\hirz\geq 3$ this corresponds to a genuine and explicit pair of  $4$-weighted extremal K\"ahler twins, $(g_{\frac{1}{\hirz}},\frac{1}{\hirz}\gz+1)$ and $(g_{\frac{1}{\hirz}},\frac{-2}{\hirz}\gz+1)$,
on the $\hirz$th Hirzebruch surface.
\item Equation \eqref{twins} is solved for $s=2/\hirz, x=1/\hirz, a=0,$ and $b=b_\hirz=-\frac{\hirz \left(\hirz^2-3\right)}{\hirz^4-\hirz^2+2}$. For $\hirz\geq 2$, we can check that $-1<b_\hirz<1$ and thus this solution corresponds to a genuine and explicit pair of $4$-weighted extremal K\"ahler twins, 
$(g_{\frac{1}{\hirz}},1)$ and $(g_{\frac{1}{\hirz}},b_\hirz\gz+1)$,
on the $\hirz$th Hirzebruch surface, where here $g_{\frac{1}{\hirz}}$ is Calabi Extremal. Note that the second Hirzebruch surface is included in this example.

\end{itemize}

\subsection{The trivial Hirzebruch Surface $\bbf_0$}\label{0hirze}
We end \S \ref{KahlerHirze} by considering the $0^{th}$ Hirzebruch surface $\bbf_0:= \bbc\bbp^1\times\bbc\bbp^1$. While product metrics on $\bbf_0$ may be considered as degenerate cases of the admissible set-up in \S \ref{hirze}, we will take a direct approach to such metrics.
Indeed a product K\"ahler structure $(g,\omega)$ on $\bbc\bbp^1\times\bbc\bbp^1$ can be described on a open dense set as follows:
\begin{equation}\label{productmetric}
g= \frac{1}{A(x)} dx^2+A(x) dt^2 + \frac{1}{B(y)} dy^2 + B(y) ds^2, \quad \omega = dx\wedge dt + dy\wedge ds,
\end{equation}
where $\alpha,\beta >0$, $-\alpha<x<\alpha$, $-\beta <y<\beta$, $t,s$ are angular coordinates, and $A(x)$, $B(y)$ are positive smooth functions. 
For $g$ to extend to all of $\bbc\bbp^1\times\bbc\bbp^1$ as a smooth K\"ahler metric we
have the following boundary conditions:
\begin{align}
\begin{split}
A(\pm \alpha)=0 \; \; & \;\; B(\pm \beta)=0\\
A'(\pm \alpha)=\mp 2\; \; & \;\; B'(\pm \beta)=\mp 2
\end{split}
\end{align}
Varying the values of $\alpha,\beta \in \bbr^+$ sweeps out the K\"ahler cone of $\bbc\bbp^1\times\bbc\bbp^1$ with product metrics. 

We know from \cite{Leb15} that with $0<\alpha<\beta<2\alpha$, $f>0$ a certain non-constant affine function of
$y$, $A(x) = (\alpha^2-x^2)/\alpha$, and $B(y)$ is an appropriate $4^{th}$ degree polynomial, the metric of the form \eqref{productmetric}
is a (non-$cscK$) $Scal_{f,4}(g)$-constant K\"ahler metric on
$\bbc\bbp^1\times\bbc\bbp^1$.

More generally and specifically, let $\alpha,\beta >0$ and $c$ be a real number such that $c^2\beta^2<(\alpha^2+\beta^2)^2$ and consider the positive linear affine function $f(y)=\alpha^2+\beta^2+cy$ for $-\beta<y<\beta$. Then
with 
$$A(x) = (\alpha^2-x^2)/\alpha$$ and 
$$B(y)=\frac{(\beta^2-y^2) \left(6 \alpha (\alpha^2 + \beta^2)^2+(\beta-\alpha) c^2 \beta^2-(\alpha + \beta) c^2y^2\right)}{2 \alpha \beta \left(3 \left(\alpha^2+\beta^2\right)^2-\beta^2 c^2\right)},$$
we get a smooth K\"ahler product metric\footnote{It is easy to check that $B(y)>0$ for $-\beta<y<\beta$, $B(\pm \beta)=0$, and that $B\,'(\pm \beta)=\mp2$.}, $g_c$, of the form \eqref{productmetric}, 
on $\bbc\bbp^1\times\bbc\bbp^1$ satisfying that $Scal_{f,4}(g)$ is a linear affine function function of $y$:
$$ 
\begin{array}{ccl}
Scal_{f,4}(g) &= &\frac{6 \left((\alpha+\beta) \left(\alpha^2+\beta^2\right)^4 -6 \alpha \beta^2 (\alpha^2 + \beta^2)^2 c^2  +(\alpha - \beta) \beta^4c^4 \right)}{\alpha \beta \left(3 \left(\alpha^2+\beta^2\right)^2-\beta^2 c^2\right)}\\
\\
&-& \frac{12c \left( (\alpha^2 + \beta^2)(  (2 \alpha - \beta)(\alpha^2 + \beta^2)^2 + \beta^3 c^2)\right)}{\alpha \beta \left(3 \left(\alpha^2+\beta^2\right)^2-\beta^2 c^2\right)}y
\end{array}
$$
This gives a family of $(f,4)$-extremal K\"ahler metrics, containing LeBrun's Einstein-Maxwell solutions  \cite{Leb15} (when $c$ is such that $Scal_{f,4}(g)$ is constant) as a special cases.
Moreover, since $B(y)$ for $c$ and $-c$ are the same, we have that $g_c=g_{-c}$ (let us denote this by $g_{|c|}$).Thus we have $4$-weighted extremal K\"ahler twins $(g_{|c|},\alpha^2+\beta^2+cy)$, $(g_{|c|},\alpha^2+\beta^2-cy)$.
One can also easily check that this is the only possibility of twinning among this family of $4$-weighted extremal K\"ahler metrics.

Note that when $\alpha$ and $\beta$ are rational, the rays corresponding to $f(y)=\alpha^2+\beta^2\pm cy$ in the unreduced Sasaki-Reeb cone of the Boothby-Wang constructed Sasaki manifold with respect to an
appropriate rescale of $g_c$ (and its K\"ahler form) are related under a transformation in the Weyl group $\bbz_2$ acting on the cone.
Furthermore, in that case, the Sasaki metric corresponding to $f(y)=\alpha^2+\beta^2+ cy$ is $cscS$ precisely when $Scal_{f,4}(g)/f$ equals a constant. From the above expression for $Scal_{f,4}(g)$, a calculation shows that this
is equivalent to $c$ satisfying the equation
$$c^2=\frac{(\beta-5\alpha)}{(\beta-\alpha)}\frac{(\alpha^2 + \beta^2)^2}{\beta^2}.$$
For $\beta>5\alpha$ this has viable solutions $c=\pm \sqrt{\frac{(\beta-5\alpha)}{(\beta-\alpha)}}\frac{(\alpha^2 + \beta^2)}{\beta}$, recovering the ``extra'' $cscS$ solutions from \cite{Leg10} on the Wang-Ziller join. From the above discussion we can see that these $cscS$ are in fact extremal Sasaki twins.

\section{Admissible Ruled surfaces over Riemann surfaces of any genus}\label{highergenus}
In this section we show that $4$-weighted extremal K\"ahler twins exist outside of the realm of Hirzebruch surfaces (and other toric K\"ahler manifolds). To do this we simply supply related  examples that can be directly derived from a generalization of the set-up in 
\S \ref{KahlerHirze}. 

Let $\Sigma_{\mathfrak g}$  be a Riemann surface of any genus ${\mathfrak g}$, let $\hirz\in \bbz^+$, and
let $g_{\Sigma_{\mathfrak g}}$ be a K\"ahler Einstein metric on $\Sigma_{\mathfrak g}$ of constant scalar curvature $2s=\frac{4(1-{\mathfrak g})}{\hirz}$, with K\"ahler form
$\omega_{\Sigma_{\mathfrak g}}$. Now let $L_\hirz\rightarrow \Sigma_{\mathfrak g}$ be a holomorphic line bundle over $\Sigma_{\mathfrak g}$ such that $c_1(L_\hirz)=[\frac{\omega_{\Sigma_{\mathfrak g}}}{2 \pi}]$.
Then $c_1(L_\hirz)=\hirz\Omega_{\mathfrak g}$ with $\Omega_{\mathfrak g}$ being the standard primitive, integer
K\"ahler class on $\Sigma_{\mathfrak g}$. Following e.g. \cite{ACGT08} we have an \emph{admissible manifold} $S_\hirz$ 
of the form $\bbp({\mathcal O} \oplus L_\hirz) \rightarrow \Sigma_{\mathfrak g}$ and, similarly to \S \ref{hirze} we can exhaust the K\"ahler cone (up to scale) with
admissible metrics $g_x$ of the form \eqref{metric}, where 
$x\in (0,1)$ still determines the K\"ahler class (up to scale) via \eqref{class1}. Further, following \cite{ApMaTF18}, for each choice of $x\in (0,1)$ and $c\in (-1,1)$ we have a function $F_{x,c}(\gz)$
defined by \eqref{wextrpol5mnf} that satisfies (ii) and (iii) of \eqref{positivityF} and would give us an 
admissible $(c\,\gz+1,4)$-extremal metric $g_{x,c}$ provided that also (i) of \eqref{positivityF}  is satisfied. It is the latter condition that is not guaranteed in general when ${\mathfrak g}>1$.
As long as we are aware of this, we can use everything else that was developed in \S \ref{hirzetwins} to find $4$-weighted extremal K\"ahler twins.

Recall that 
$$F_{x,c}(\mathfrak{z})=(1-\gz^2)P_{x,c}(\gz),$$
where $P_{x,c}(\gz)$ is given by \eqref{polP}.
Thus, (i) of \eqref{positivityF} is equivalent to $P_{x,c}(\gz) >0$ for $-1<\gz <1$.
Now we notice that if $c=x$, then $P_{x,c}(\gz) = (1+x \gz)$, which is clearly positive for $-1<\gz <1$ since $x\in (0,1)$.
Thus any solution $(a,b)$ to \eqref{twins} with $a=x$ such that $-1<b<1$ would produce a pair of $4$-weighted extremal K\"ahler twins.
Setting $a=x$ in \eqref{twins} gives us formally a solution $b\neq a$.
\begin{equation}\label{bhighergenusex}
b=\frac{x (2 - s x)}{3 x^2-s x-1}.
\end{equation}
Thus, for $-1<b<1$ to be true, $s$ and $x$ need to satisfy that $-1<\frac{x (2 - s x)}{3 x^2-s x-1}<1$. Clearly there are many choices for $s=\frac{2(\mathfrak g -1)}{\hirz}$ and $x\in (0,1)$ where this is satisfied and thus
we have many examples of $4$-weighted extremal K\"ahler twins.

To make this explicit and tie it in with extremal Sasaki twins we will find solutions that correspond to smooth $M_{\mathfrak g}\ast_{l_1,l_2}S^3_\bfw$ joins as described in e.g. \cite{BoTo14a,BoTo13} where
$M_{\mathfrak g}$ is the canonical Sasaki manifold over $\Sigma_{\mathfrak g}$. This means that we want to cast $S_\hirz$ with the K\"ahler class determined by $x$ via
\eqref{class1} as the regular quotient of $M_{\mathfrak g}\ast_{l_1,l_2}S^3_\bfw$. This process is discussed in \S 4.4 of \cite{BHLT16}. 
Since we want to pick cases where $M_{\mathfrak g}\ast_{l_1,l_2}S^3_\bfw$ is a $S^3$-bundle over 
$\Sigma_{\mathfrak g}$ we will chose $l_2=1$ in accordance with Proposition 7.6.7 in \cite{BG05}. We also want to arrive at examples for both the twisted and 
untwisted bundle of each possible genus $\mathfrak g$.

Following e.g. \S 4.4 of \cite{BHLT16} we see that the correspondence between the $M_{\mathfrak g}\ast_{l_1,1}S^3_\bfw$  and $S_\hirz$ with K\"ahler class determined by $x$ is as follows
\begin{equation}
\hirz=l_1(w_1-w_2) \quad x=\frac{w_1-w_2}{w_1+w_2}
\end{equation}
and by Theorem 3.5 in \cite{BoTo13} we know that $M_{\mathfrak g}\ast_{l_1,1}S^3_\bfw$ is diffeomorphic
to the trivial (untwisted) bundle  $\Sigma_{\mathfrak g} \times S^3$ if $\hirz$ is even and diffeomorphic to the non-trivial (twisted) $S^3$-bundle over $\Sigma_{\mathfrak g}$ if $\hirz$ is odd.

We look at four cases:
\begin{enumerate}
\item Untwisted:
\begin{enumerate}
\item $w_1=11$, $w_2=9$ and $l_1=1$ giving $x=1/10$, $\hirz=2$, and hence $s=(1-\mathfrak g)$
\item $w_1=11$, $w_2=9$ and $l_1=2$ giving $x=1/10$, $\hirz=4$, and hence $s=\frac{(1-\mathfrak g)}{2}$
\end{enumerate}
\item Twisted
\begin{enumerate}
\item $w_1=51$, $w_2=50$ and $l_1=1$ giving $x=1/101$, $\hirz=1$, and hence $s=2(1-\mathfrak g)$
\item $w_1=51$, $w_2=50$ and $l_1=3$ giving $x=1/101$, $\hirz=3$, and hence $s=\frac{2(1-\mathfrak g)}{3}.$
\end{enumerate}
\end{enumerate}

We now see that for each non-negative integer value of $\mathfrak g$, there is at least one untwisted and one twisted sub-case above where $b$ in \eqref{bhighergenusex} satisfies $-1<b<1$..
\begin{enumerate}
\item Untwisted:
When $x=1/10$ we get
$$b=\frac{s-20}{10 s+97}.$$
\begin{enumerate}
\item For $s=(1-\mathfrak g)$, this becomes
$$b=\frac{\mathfrak g +19}{10 \mathfrak g-107}.$$
\item For $s=\frac{(1-\mathfrak g)}{2}$, this becomes
$$b=\frac{\mathfrak g +39}{2(5 \mathfrak g-102)}.$$
\end{enumerate}
It is a simple exercise to check for all non-negative integer values of $\mathfrak g$ at least one of
$\frac{\mathfrak g +19}{10 \mathfrak g-107}$ and $\frac{\mathfrak g +39}{2(5 \mathfrak g-102)}$ will be inside the interval $(-1,1)$.

\item Twisted:
When $x=1/101$ we get
$$b=\frac{s-202}{101 s+10198}.$$
\begin{enumerate}
\item For $s=2(1-\mathfrak g)$, this becomes
$$b=\frac{\mathfrak g +100}{101 \mathfrak g-5200}.$$
\item For $s=\frac{2(1-\mathfrak g)}{3}$, this becomes
$$b=\frac{\mathfrak g +302}{(101 \mathfrak g-15398)}.$$
\end{enumerate}

Again, it is easy to check that for all non-negative integer values of $\mathfrak g$ at least one of
$\frac{\mathfrak g +100}{101 \mathfrak g-5200}$ and $\frac{\mathfrak g +302}{(101 \mathfrak g-15398)}$ will be inside the interval $(-1,1)$.
\end{enumerate}

We can now conclude with the following proposition:
\begin{proposition}
For any genus  $\mathfrak g$ there is choice of an even (as well as odd) $\hirz\in \bbz^+$ such that the ruled surface
$S_\hirz=\bbp({\mathcal O} \oplus L_\hirz) \rightarrow \Sigma_{\mathfrak g}$ has a pair of non-trivial $4$-weighted extremal K\"ahler twins where
the K\"ahler class of the shared K\"ahler metric can be taken to be a primitive integer cohomology class.

Further, the even $\hirz$, together with an appropriate K\"ahler class, can be chosen such that this gives a pair of extremal Sasaki twins on $\Sigma_{\mathfrak g} \times S^3$ 
and 
the odd $\hirz$, together with an appropriate K\"ahler class, can be chosen such that this gives a pair of extremal Sasaki twins on the non-trivial $S^3$-bundle over $\Sigma_{\mathfrak g}$.
In both cases, the $CR$-structure arises from a smooth join $M_{\mathfrak g}\ast_{l_1,1}S^3_\bfw$.
\end{proposition}

\section{Extremal Twins in the Sasaki Setting} \label{sastwins}
Let $(M,\cald, J)$ be a pseudo-convex CR structure of Sasaki type on a compact smooth manifold $M$ of dimension $2n+1$. \emph{Extremal Sasaki twins} of $(\cald,J)$ are two extremal Sasaki structures both compatible with $(\cald,J)$ and having commuting Sasaki-Reeb vector fields. In contrast to the K\"ahler setting, the twinning phenomenon  in the Sasaki setting gives rise to distinct extremal Sasaki metrics which, however, belong to the same underlying pseudo-convex CR structure, that is, one does not deform within the isotopy class.

Let $(M,\cald,J,\xi_0)$ and $(M,\cald,J,\xi_1)$ be such twins (here $\xi_i$ denotes the Sasaki-Reeb vector field, together with $(\cald,J)$, determines also the contact $1$-form $\eta_i$ (by $\cald=\ker \eta_i$, $\eta_i(\xi_i)=1$ and $\call_{\xi_i}\eta_i=0$) and the Sasaki metric $g_i$. Since the associated Sasaki-Reeb vector fields commute, they induce the action of a torus which is a subgroup of a maximal torus, say $T\subset (\mbox{Con}(\eta_0)\cap \mbox{Con}(\eta_1) )\cap \mbox{CR}(\cald,J)$ with Lie algebra $\gt$ and where $\mbox{Con}(\eta_i)$ is the strict contactomorphism group with respect to $\eta_i$ and
$\mbox{CR}(\cald, J)$ is the $CR$-transformation group.

By \cite{BGS06}, $(\cald,J,\eta_0)$ and $(\cald,J,\eta_1)$ are both $T$-invariant and we denote by
$$\mu : M\ra  \kt^*$$ 
the $\eta_0$-momentum map where, of course, $\gt^*$ is the dual of $\gt$. That is, $\mu(a)= \eta_0(\underline{a})$ where $\underline{a}$ is the vector field induced on $M$ from $a\in \kt$ via the action. Recall that $P:= \mu(M)$ is a compact polytope in the affine space $\Pi_{\xi_0}:=\{ x\in \kt^*\,|\, \langle \xi_0, x\rangle =1\}$ \cite{Ler02a}. From the work of \cite{ApCa18}, based on the Lee and Tanno formula \cite{lee86,tanno89}, we know that $(M,\cald,J,\eta_0)$ and $(M,\cald,J,\eta_1)$ are extremal Sasaki twins if and only if, denoting $f:=\eta_0(\xi_1)$,       
\begin{equation}\label{eqExtSasTwins}
\begin{split}
 &Scal(g_0)\;\; \mbox{ and }\\
 &f Scal(g_1)  = f^2 Scal(g_0) -2(n+1)f\Delta^{g_0} f -(n+1)(n+2) |df|^2_{g_0} 
\end{split}
\end{equation} are pull-backs, by $\mu$, of affine functions on $\Pi_{\xi_0}$. Indeed, $g_1$ is extremal if and only if $Scal(g_1)= \eta_1(Z_1^{ext})$ for the contact, Killing vector field $Z_1^{ext}$, so that $Z_1^{ext} \in \kt$. Now,  $\eta_1 = \eta_0(\xi_1)^{-1}\eta_0 = f^{-1}\eta_0$. Thus  $g_1$ is extremal if and only if $Scal(g_1)= \eta_1(Z_1^{ext})= f^{-1}\eta_0(Z_1^{ext})$, that is iff $f Scal(g_1)= \eta_0(Z_1^{ext}) = \langle\mu, Z_1^{ext}\rangle$.

In particular, noticing that $f$ is also the pull-back of an affine function on $\Pi_{\xi_0}$, if $(M,\cald,J,\eta_0)$ and $(M,\cald,J,\eta_1)$ are extremal Sasaki twins, then the function 
\begin{equation} \label{SasakiExtTwinsPOl} fScal(g_1)  - f^2 Scal(g_0)= -2(n+1)f\Delta^{g_0} f -(n+1)(n+2)|df|^2_{g_0}
\end{equation} 
(seen as a function on $P$) is the pull-back of a polynomial function of degree at most $3$ on $\Pi_{\xi_0}$. \\

{\it In what follows the Killing potential $f = \eta_0(\xi_1)$ is often identified with the affine function on $\kt^*$ whose pull-back by $\mu$ is $f$.} 
\bigskip

In the next two subsections, we will exploit the affine geometry and basic facts on Hamiltonian torus action to study the extremal Sasaki twins whenever one of the twins has constant scalar curvature. To this end, we introduce more notation and recall that the moment polytope $P= \mu(M)$ inherits a Duistermaat-Heckman measure from the pushforward ${\mbox{DH}} := \mu_* \eta_0\wedge (d\eta_0)^n/n!$ and we denote by $x_o=x_o^{\xi_o} \in P$ the barycenter of $P$ with respect to this measure. We get dual decompositions \begin{equation}\label{eqDualDecomp}
    \kt^*= \bbr x_o \oplus \xi_0^0\;\;\; \mbox{ and }\;\; \kt := \bbr \xi_0\oplus \check{\kt}
\end{equation} where $\check{\kt} := \{a\in \kt \,|\, \langle x_o, a\rangle =0 \} = \{a\in \kt \,|\, \int_M \eta_0(\underline{a}) \eta_0 \wedge (d\eta_0)^n =0 \}$. 

\subsection{Sasaki-extremal twins of Sasaki-Einstein metrics}\label{twinsofSE}

 We now assume furthermore that $(M, \cald,J,\eta_0,\xi_0, g_0)$ is a Sasaki-Einstein manifold. Up to an inconsequential
 transverse homothety (see (7.3.10) in \cite{BG05}), we may assume that $Scal(g_0)=2n$. From the Sasaki version of the Matsushima--Lichnerowicz theorem, the Sasaki-Einstein condition implies that Killing potentials are eigenfunctions of the Laplacian. In our setting this means that for each affine function $f\in \mbox{Aff}(\kt^*)$ 
  $$\Delta^{g_0} f =2 (f-f(x_o)).$$ 
 
 The condition that $(M, \cald,J,\eta_1,g_1)$ is a twin is equivalent to the condition that $fScal(g_1)$ is the pull-back of an affine-linear function on $P$ where $f := \eta_0(\xi_1)$. So we write $$fScal(g_1) = \mu^*\ell_1^{ext}$$ and $f(x)= \langle v,x \rangle +\lambda$, where $\lambda=f(x_o)$. Using that $(M, D,J,\eta_0,\xi_0, g_0)$ is a  Sasaki-Einstein manifold and the Lee--Tanno formula \eqref{eqExtSasTwins}, we then have, as functions on $P$
\begin{equation}\label{eqfscaltoricFano0}
\begin{split}
   \ell_1^{ext} &= f^2 Scal(g_0) -2(n+1)( f\Delta^{g_0}(f\circ \mu)-2(n+1)(n+2)  |df|^2_{g_0}\\ 
   &= 2n f^2-4 (n+1) f(f-\lambda) - 2(n+1)(n+2) |df|^2_{g_0}\\
   &= -2(n+2) f^2 +4 \lambda (n+1) f - 2(n+1)(n+2)|df|^2_{g_0}.
\end{split}
\end{equation} Observe that the last term of \eqref{eqfscaltoricFano0} vanishes on the critical points of $\mu$. Since the vertices of $P$ are images by $\mu$ of critical points of $\mu$ we see that equation~\eqref{eqfscaltoricFano0} implies that $$\ell_1^{ext}(p)= -2(n+2) f^2(p) +4 \lambda (n+1) f(p)   \;\;\; \forall p\in \mbox{vertices(P)}.$$
 
Let $k=\dim \Pi_{\xi_o}$, so that $k+1=\dim \kt$ and pick $k+1$ vertices say $p_0,\dots, p_k$ of $P$ forming a corner of $P$, i.e each segment between $p_i$ and $p_0$ is an edge of $P$ for $i=1,\dots k$.   This provides affine coordinates $\alpha \in \bbr^{k+1}$ for the hyperplane $\Pi_{\xi_o}$ so that each $p\in \Pi_{\xi_o}$ corresponds to $p= \sum_{i=0}^k\alpha_ip_i$ with $\sum_{i=0}^k\alpha_i=1$. In these coordinates an affine-linear function $f\in \mbox{Aff}(\Pi_{\xi_o},\bbr)$ corresponds to a $(k+1)$-tuple $f= (w_i)_{i=1,\dots, k}$ so that $$f(p)= \sum_{i=0}^k w_i\alpha_i.$$

\begin{lemma} \label{lem:CombinCdt} Let $(M, \cald,J,\eta_0,\xi_0, g_0)$ be a Sasaki-Einstein manifold with moment polytope $P \subset H$. In terms of the affine coordinates above, if an affine-linear function $f =(w_i)_{i=1,\dots, k} \in \mbox{Aff}(\Pi_{\xi_o},\bbr_{>0})$ corresponds to an extremal Sasaki twin of $(M, \cald,J,\eta_0,\xi_0, g_0)$ then for every vertex $p=(\alpha_i)_{i=0,\dots,k}$ of $P$ we have  
\begin{equation}\label{eqfscaltoricFano04}
 \sum_{i=0}^k \alpha_iw^2_i  = \left(\sum_{i=0}^k \alpha_iw_i \right)^2. 
\end{equation}      
\end{lemma} 
Note that for $p \in \{ p_0,\dots, p_k \}$, \eqref{eqfscaltoricFano04} is trivially satisfied. Moreover the equation $\sum_{i=0}^k \alpha_i=1$ implies that \eqref{eqfscaltoricFano04} is equivalent to 
\begin{equation}\label{eqfscaltoricFano05}
 \sum_{i=1}^k \alpha_i(w_i -w_0)^2 = \left(\sum_{i=1}^k \alpha_i(w_i-w_0) \right)^2. 
\end{equation}
\begin{proof} In these coordinates \eqref{eqfscaltoricFano0} reads 
\begin{equation}\label{eqfscaltoricFano0verticesBase}
   \ell_2^{ext} (p_i)=  -2(n+2) f^2(p_i) +4 \lambda (n+1) f(p_i) =  -2(n+2)w^2_i  +4 \lambda(n+1)w_i 
\end{equation} for $i=0,\dots, k$. For any other vertex $p\in P$ with affine coordinates $(\alpha_{i})_{i=0,\dots, k}$ we have 
\begin{equation}\label{eqfscaltoricFano02}
\ell_2^{ext} (p) = \sum_{i=0}^k \alpha_i ( -2(n+2)w^2_i  +4 \lambda(n+1)w_i ) 
\end{equation}
and, on the other hand, 
\begin{equation}\label{eqfscaltoricFano03}
\begin{split}
   \ell_2^{ext}(p)  &= -2(n+2) f^2(p) +4 \lambda (n+1) f(p) \\
   &= -2(n+2) (\sum_{i=0}^k\alpha_iw_i )^2 +4\lambda (n+1) \sum_{i=0}^k \alpha_iw_i. 
\end{split}
\end{equation} The last two equations coincide and the last terms cancel each other. Therefore, if $f =(w_i)_{i=0,\dots,k} \in \mbox{Aff}(P,\bbr_{>0})$ corresponds to a twin, then for each vertex $p=(\alpha_i)_{i=0,\dots,k}$ of $P$ equation \eqref{eqfscaltoricFano04} holds.
\end{proof}

The next two examples show that the condition obtained above is non-trivial in the toric case that is when the torus has maximal dimension (i.e $k =n = (\dim M-1)/2$). 

\begin{example}We check the condition \eqref{eqfscaltoricFano04} on the standard Sasaki--Einstein structure on the circle bundle over $\bbc\bbp^1\times \bbc\bbp^1$. The vertices of $P$ are $p_0=(-1,-1)$, $p_1=(-1,1)$, $p_2=(1,-1)$ and $p=(1,1)$, so that $\alpha_0 =-1$, $\alpha_1=1$ and $\alpha_2=1$. The solutions (twins) we seek correspond to solutions $(w_i)_{i=0,\dots, 2}\in \bbr_{>0}^3$ satisfying 
$$-w_0^2 +w_1^2 + w_2^2 = \left(-w_0 +w_1 +w_2\right)^2.$$ 
So $w_0 = w_1 = w_2=1$ is the solution corresponding to the original Sasaki--Einstein metric (i.e $f\equiv 1$) and one can check easily that 
$$-w_0^2 +w_1^2 + w_2^2 = \left(-w_0 +w_1 +w_2\right)^2 \;\Longleftrightarrow \; 0=(w_1-w_0)(w_2-w_0).$$ This implies that, up to a dilatation, extremal Sasaki twins of the Sasaki--Einstein metric on ${\mathbb S}^{3} \times {\mathbb S}^{3}$ can only correspond to $w_0 =  w_1 =1$ or $w_0 = w_2=1$. In fact, we will see later in Corollary \ref{SEonS3xS3hasnotwin} that this Sasaki--Einstein metric has no extremal Sasaki twins at all.
\end{example}

\begin{example}\label{exBl3cp2} We check the condition \eqref{eqfscaltoricFano04} on the standard Sasaki--Einstein structure on the circle bundle over $Bl_3(\bbc\bbp^2)$, the blow-up of $\bbc\bbp^2$ at the $3$ generic fixed points. The Delzant polytope of the anticanonical bundle of $Bl_3(\bbc\bbp^2)$ is the hexagon with vertices $p_0=(-1,0)$, $p_1=(0,-1)$, $p_2=(-1,1)$, $p_3=(0,1)$, $p_4=(1,0)$, $p_5=(1,-1)$. Taking $p_0,p_1,p_2$ as the base points of the barycentric coordinates we get that the coordinates of three other vertices are $p_3= (-2,1,2)$, $p_4=(-3,2,2)$ and $p_5=(-2,2,1)$.  The condition \eqref{eqfscaltoricFano04} on $p_3$ gives that the function $f=(w_0,w_1,w_2)$ must satisfy \begin{equation}\label{eqVerticesCdtP3}
    -2(w_2-w_0)^2 = 4(w_2-w_0)(w_1-w_0)
\end{equation} while on $p_4$ the condition \eqref{eqfscaltoricFano04} is \begin{equation}\label{eqVerticesCdtP4}
    -((w_1-w_0)^2 +(w_2-w_0)^2)= 4(w_1-w_0)(w_2-w_0).
\end{equation}  Inserting \eqref{eqVerticesCdtP3} into \eqref{eqVerticesCdtP4} gives that $$(w_1-w_0)^2 = (w_2-w_0)^2$$ which in turns gives $(w_1-w_0) = \pm (w_2-w_0)$. Putting this back in \eqref{eqVerticesCdtP3} implies that either     
    $$3(w_1-w_0)^2 = 9 (w_1-w_0)^2 \qquad \mbox{ or } \qquad 3(w_1-w_0)^2 = (w_1-w_0)^2$$ which are both impossible unless $w_0= w_1$ and thus $w_0= w_1=w_2$. In conclusion, the Sasaki--Einstein structure on the circle bundle over $Bl_3(\bbc\bbp^2)$ has no twin.    
\end{example}

\subsection{cscS twins}\label{ssCscS}
If both $Scal(g_0)$ and $Scal(g_1)$ are constant, say $C_0:=Scal(g_0)$ and $C_1:= Scal(g_1)$, then dividing \eqref{SasakiExtTwinsPOl} by $f$, gives \begin{equation} \label{SasakiCsCTwinsPOl} C_1  - f C_0 = -2(n+1)\Delta^{g_0} f -(n+1)(n+2)f^{-1}|df|^2_{g_0}.\end{equation} 

\begin{lemma} Let $(M,\cald,J)$ is a pseudoconvex compact connected manifold on which acts a compact torus $T$ and denote $\kt^+$ the Sasaki cone. \begin{itemize}
    \item[1)] If $(M,\cald,J,\xi_0)$ and $(M,\cald,J,\xi_1)$ are non-colinear cscS twins, then there exists no non-zero $\lambda \in \bbr$ such that $\lambda\xi_0+\xi_1\in \kt^+$ and $(M,\cald,J,\lambda\xi_0+\xi_1)$ is cscS. 
    \item[2)] Moreover, if $(M,\cald,J,\xi_1),\dots, (M,\cald,J,\xi_k)$ are cscS twins, then there is no cscS on $(M,\cald,J)$ in the cone spanned by $\xi_1,\dots, \xi_k$ in $\kt^+$ unless the $\xi_1,\dots,\xi_k$ are all colinear.  
    \end{itemize}
\end{lemma}
\begin{proof} We first prove the second claim. Assume that the combination $\xi= \sum_{i=1}^k\lambda_i \xi_i$, where $(\lambda_1,\dots,\lambda_k)\in \bbr^k_{>0}$, defines a cscS structure on $(\cald,J)$ and denote by $\eta$ the contact $1$-form and by $g$ the Sasaki metric associated to $(\cald,J,\xi)$. The hypothesis is that each $\xi_i$ defines a cscS structure on $(\cald,J)$ and denote $f_i:= \eta(\xi_i)$ the associated Killing potential and $C_i \in \bbr$ the constant transversal scalar curvature of $(\cald,J,\xi_i)$. Since $\eta(\xi)=1$ we have $$1= \sum_{i=1}^k\lambda_i f_i$$ and, moreover, each function $f_i$ satisfies the equation
\begin{equation} \label{SasakiCsCTwinsPOlicone} C_i  -Cf_i =  -2(n+1)\Delta^{g} f_i -(n+1)(n+2)(f_i)^{-1}|df_i|^2_{g}\end{equation} for $i=1,\dots, k$. Taking the weighted sum over $\lambda_i$, we have 
\begin{equation} \label{SasakiCsCTwinsPOlicone2} \sum_{i=1}^k\lambda_iC_i -C  = +0 -(n+1)(n+2)\sum_{i=1}^k\frac{\lambda_i}{f_i}|df_i|^2_{g}\end{equation} for $i=1,\dots, k$. Since the left hand side of \eqref{SasakiCsCTwinsPOlicone2} is constant, the right hand side is constant as well. But the right hand side vanishes on the maximally fixed point set of the torus action (which is non-empty since $M$ is compact) and then $$\sum_{i=1}^k\frac{\lambda_i}{f_i}|df_i|^2_{g}\equiv 0.$$ This implies that $df_i\equiv 0$, that is $f_i$ is constant, for $i=1,\dots, k$, since $(\lambda_1,\dots,\lambda_k)\in \bbr^k_{>0}$. This proves the second claim. 

For the first claim, we get two equations 
$$C_1 -Cf =  -2(n+1)\Delta^{g} f-(n+1)(n+2)(f)^{-1}|df|^2_{g}$$ 
and
\begin{equation} \label{SasakiCsCTwinsPOlambda} C_\lambda  -C(f+\lambda) =  -2(n+1)\Delta^{g} f-(n+1)(n+2)(f+\lambda)^{-1}|df|^2_{g}.\end{equation} So $C_\lambda -C_1 -C\lambda = (n+1)(n+2)\frac{\lambda}{f(f+\lambda)}|df|^2_{g}$ which holds if and only if $\lambda=0$ or $df\equiv 0$. \end{proof}

\begin{remark} Another way to understand the first claim of the last Lemma goes as follows: given a cscS $(M,\cald,J,\xi_0)$ and $\xi_1 \in\kt^+$ not colinear to $\xi_0$, there is at most one $\lambda\in \bbr$ such that $(M,\cald,J,\lambda\xi_0+\xi_1)$ is likely to be cscS. This candidate is the unique critical point of the Einstein-Hilbert functional ${\bf EH}: \kt^+ \ra \bbr$ along the path $\lambda \mapsto \lambda\xi_0+\xi_1$ thanks to \cite[Lemma 3.3]{BHLT15}. \end{remark}

\begin{corollary} Let $(M,\cald,J)$ is a pseudoconvex compact connected manifold on which acts a compact torus $T$ and denote $\kt^+$ the Sasaki cone. Denote by $\mathrm{cscS}(\kt^+) \subset \kt^+$ the set of rays of cscS structure (i.e for each $\xi \in \mathrm{cscS}(\kt^+)$, $(M,\cald,J,\xi)$ is cscS). Then, 
\begin{itemize}
    \item[1)] $\mathrm{cscS}(\kt^+)/\bbr_{>0} \subset \kt^+/\bbr_{>0}$ meets every line in $\kt^+/\bbr_{>0}$ in at most two points;    
    \item[2)] $\mathrm{cscS}(\kt^+)$ lies in the boundary of its convex hull. 
\end{itemize}
\end{corollary}

By \cite{BHLT21}, the set  $\mathrm{cscS}(\kt^+)/\bbr_{>0} \subset \kt^+/\bbr_{>0}$ is algebraic, so the last observation gives that its degree is at most two. 

\begin{corollary} \label{corQuadraticVar} The set $\mathrm{cscS}(\kt^+)/\bbr_{>0}$ is either empty or lies in a quadratic subvariety of $\kt^+/\bbr_{>0}$. 
\end{corollary}

\subsection{Toric extremal Sasaki twins}
We now assume that $(M,\cald,J,\eta_0)$ and $(M,\cald,J,\eta_1)$ are toric meaning that the maximal torus $$T\subset (\mbox{Con}(\eta_0)\cap \mbox{Con}(\eta_1) )\cap \mbox{CR}(\cald,J)$$ has dimension $n+1$ where $\dim M= 2n+1$.

In that case, the Delzant-Lerman-Tolman correspondence, adapted to the contact setting \cite{BG00b, Ler02a}, states that the contact toric strict manifold $(M,\cald,\eta_0)$ is completely determined by the moment polytope $$P:=\mu(M)$$ and a {\it labelling}, that is a set of affine-linear functions $\ell=\{\ell_1,\dots, \ell_d\} \subset \mbox{Aff}(\Pi_{\xi_0}, \bbr)$ which itself corresponds to the weights of the action on $M$. Recall that $F_i=\ell_i^{-1}(0)\cap P$ are the facets of the polytope $P$ and that 
$$P= \{x\in \Pi_{\xi_0}\,|\, \ell(x) \geq 0\}$$ 
where we denote $u_i := d\ell_i$ the (inward) normal vector to the facet $F_i$.

Any $T$--invariant tensor or operator on $(M,\cald,J,\eta_0)$ can be suitably translated in terms of the data $(P,\ell, \Lambda)$ where $\Lambda$ is the lattice of weights of $T$ in $\kt^*$, this is the essence of the (Sasaki version) of the Abreu–Guillemin theory (see eg \cite{Apostolov19} or \cite{Abr09,Leg10} for the Sasaki version). Therefore, the CR-structure, equivalently the Sasaki metric $g_0$, corresponds to a {\it symplectic potential} $ \varphi  \in C^\infty(\mathring{P})\cap C^0(P)$ fulfilling the following conditions 
\begin{itemize}
    \item $\varphi$ is strictly convex on the relative interior of the polytope $\mathring{P}$, the inverse of its Hessian $${\bf H} := (H_{i,j})_{i,j=1, \dots n} = \left(\frac{\partial^2 \varphi }{\partial x_i\partial x_j} \right)^{-1}$$ extends smoothly on $P$ and satisfies 
    \item ${\bf H} (u_k,\cdot) = 0$ on $F_k$;
    \item $ d {\bf H} (u_k,u_k) = 2u_k$ on $F_k$.
\end{itemize} Any $T$--invariant function $h$ on $M$ is the pull-back of a function on $P$ and to lighten the notation we will directly denote $h(x)$ the value taken by $h$ on $\mu^{-1}(x)$ where $x =(x_1,\dots, x_n)$ denotes affine coordinates on $\Pi_{\xi_0}$.   

The transverse scalar curvature of $(M,\cald,J,\eta_0)$ is $T$--invariant and is 
\begin{equation}\label{eqScaltoric}
    Scal(g_0) = -\frac{\partial^2 H_{ij}}{\partial x_i\partial x_j} 
\end{equation} and, $(M,\cald,J,\eta_0)$ is extremal if and only if \eqref{eqScaltoric} is affine-linear. In that case, this affine-linear function, say $\ell^{ext}$ is completely determined by the data $(P,\ell)$. The basic Laplacian on $T$-invariant functions has the following form 
\begin{equation}\label{eqLaptoric}
    \Delta^{g_0}_B = - \frac{\partial }{\partial x_i} \left( H_{ij}\frac{\partial }{\partial x_j}\right) 
\end{equation} while the square norm of the $g_0$-gradient of a $T$-invariant function $h$ is 
\begin{equation}\label{eqNormgradtoric}
    |dh|^2_{g_0} =  \sum_{i,j=1}^n H_{ij} \frac{\partial h }{\partial x_i} \frac{\partial h}{\partial x_j}. 
\end{equation} Now the fact that $f$ is $g_0$-Killing and positive corresponds to the fact that $f \in\mbox{Aff}(\Pi_{\xi_0}, \bbr_{>0})$ and we write it as   
 $$f(x) = \langle v, x \rangle +\lambda = \sum_{i=1}^nv_i x_i +\lambda.$$      

Applying formula \eqref{eqExtSasTwins}, like in \cite{ApCa18}, and the formulas above, we have that $(M,\cald,J,\eta_0)$ and $(M,\cald,J,\eta_1)$ are both extremal Sasaki if and only if $Scal(g_0)=\ell^{ext}(x)$ and the following expression  
\begin{equation}
\begin{split}
    &f(x)^2 \ell^{ext}(x)  -2(n+1)f(x)\Delta^{g_0} f -(n+1)(n+2) |df|^2_{g_0}=\\
    &\quad\quad =f(x)^2 \ell^{ext}(x) + 2(n+1)f(x) \sum_{i,j=1}^n  H_{ij,i}v_j   - \sum_{i,j=1}^n (n+1)(n+2) H_{ij}v_iv_j   
\end{split}
\end{equation} is affine-linear. In particular, in that case, 
\begin{equation}\label{eqExttoric0} 2(n+1)f(x) \sum_{i,j=1}^n  H_{ij,i}v_j   - \sum_{i,j=1}^n (n+1)(n+2) H_{ij}v_iv_j\end{equation} is a polynomial of degree $3$ in $x$. Note that following \eqref{SasakiExtTwinsPOl}, we also know that \eqref{eqExttoric0} should be divisible by $f$.

\subsection{The toric odd-dimensional sphere}\label{ssSIMPLEX}
Here we consider the case of the odd-dimensional sphere $(M={\mathbb S}^{2n+1}, \cald, J,\eta_0)$ with its standard CR flat, toric, Sasaki-Einstein structure and whose K\"ahler quotient is ${\mathbb C}{\mathbb P}^n$ with the Fubini-Study structure. From Bryant \cite{Bry01} and Apostolov--Calderbank--Gauduchon \cite{ApCaGa06}, we know that the extremal K\"ahler metrics on the weighted projectives spaces are all Bochner flat and K\"ahler quotients of the CR flat Sasaki sphere. It implies, in our language, that the Sasaki cone of the CR flat sphere is exhausted by (twinning) extremal Sasaki metrics.  \\

In this subsection, we retrieve this result using the toric set-up and basic calculus. In that case the polytope $P$ is a $n$-simplex and we can assume that $\ell_1(x) = 1+x_1$, $\dots$, $\ell_n(x) = 1+x_n$ and $\ell_0(x) = 1-x_1-\dots-x_n$. The symplectic potential is known, from \cite{Abr01} to be $$\varphi(x) = \frac{1}{2}\sum_{k=0}^n \ell_k(x) \log \ell_k(x)$$ and, thus, its Hessian is 
$${\bf G} = \frac{1}{2\ell_0} \begin{pmatrix}
    \frac{\ell_0}{\ell_1} +1 & 1 &1 &\dots &1\\
    1&\frac{\ell_0}{\ell_1} +1 & 1& \dots &1\\
    1&1 & \dots & \dots &1\\
    && \vdots&&\\
     1&&&1& \frac{\ell_0}{\ell_n} +1 \\
\end{pmatrix}$$ with inverse 

$${\bf H} = \begin{pmatrix}
    2\ell_1 -\frac{2\ell_1^2}{n+1} & -\frac{2\ell_1\ell_2}{n+1} &  -\frac{2\ell_1\ell_3}{n+1} &\dots & -\frac{2\ell_1\ell_n}{n+1} \\
     -\frac{2\ell_1\ell_2}{n+1} &2\ell_2 -\frac{2\ell_2^2}{n+1} &  -\frac{2\ell_3\ell_2}{n+1}& \dots & -\frac{2\ell_n\ell_2}{n+1}\\
    \vdots&& \vdots&&\vdots\\
    - \frac{2\ell_1\ell_n}{n+1} &\dots && -\frac{2\ell_{n-1}\ell_n}{n+1}& 2\ell_n -\frac{2\ell_n^2}{n+1}\\
\end{pmatrix}.$$ 

We can check the well-known fact that, for K\"ahler--Einstein structures (with Einstein constant\footnote{For the Sasaki-Einstein case, 
the transverse scalar curvature is $4n(n+1)$ and hence the Einstein constant 
of the transverse K\"ahler structure is $2(n+1)$.
However, since everything works fine up to scale (homothety deformation of the Sasaki structure) we use (as in \S \ref{twinsofSE}) the Einstein constant $1$.} $1$) the momentum map coordinates are eigenfunctions of the Laplacian and we find, for all $j=1,\dots, n$,
\begin{equation}\label{eqLaptoricCPn}
    \Delta^{g_0}_B x_j = - \sum_{i=1}^n H_{ij,i} = 2 (x_j-p_i)
\end{equation} where $p_i \in\bbr$ is a constant, the $i$-th coordinate of the barycenter of the moment polytope, so that $Scal(g_0) = 2n$.  

Replacing the values of $H_{ij}= 2 \delta_{ij}\ell_i -2\frac{\ell_i\ell_j}{n+1}$ and \eqref{eqLaptoricCPn} in \eqref{eqExttoric0} we get 
$$-4 (n+1)f(x) \sum_{j=1}^n  (x_j-p_j) v_j   - 2(n+1)(n+2) \sum_{i,j=1}^n \left(2 \delta_{ij}\ell_i -2\frac{\ell_i\ell_j}{n+1}\right))v_iv_j,$$ which is a polynomial of degree $2$ in $x$ as soon as $f \in \mbox{Aff}(\bbr^n,\bbr)$. When $f \in \mbox{Aff}(\bbr^n,\bbr_{>0})$, it defines a CR-twist $(M={\mathbb S}^{2n+1}, \cald, J,\eta_1)$ which is extremal if and only if    
\begin{equation}\label{eqfscaltoric}
\begin{split}
    fScal(g_0) = 2nf^2&(x) -4(n+1) f(x) \sum_{j=1}^n  (x_j-p_j) v_j  \\
    &- 2(n+1)(n+2) \sum_{i,j=1}^n \left(2 \delta_{ij}\ell_i(x) -2\frac{\ell_i(x)\ell_j(x)}{n+1}\right))v_iv_j
\end{split}
\end{equation} is affine-linear in $x$. Denoting $m = \sum_{i=1}^n v_i$ the right hand side of \eqref{eqfscaltoric} is
\begin{equation}\label{eqfscaltoric2}
\begin{split}
    fScal(g_1) &= 2nf^2(x) -4(n+1) f(x)( f(x) -\lambda)\\
    &\qquad  +2(n+2) (f^2(x) + 2(m-\lambda) f(x) +(m-\lambda)^2) \\
    & \qquad\qquad\qquad- 2(n+1)(n+2) \sum_{i=1}^n v_i^2(1+x_i).
\end{split}
\end{equation} Note that in the last equation, the coefficient of $f^2$ sum to zero and therefore \eqref{eqfscaltoric2} is linear in $x$. This gives a new proof of the following result which is an immediate consequence of \cite{Bry01,ApCaGa06}.

\begin{proposition} The Sasaki cone of the CR flat sphere is exhausted by extremal Sasaki metrics (that are all twins).    
\end{proposition}

\section{Twins of Toric extremal K\"ahler metrics on quadrilaterals }\label{quadrilaterals}
In this section we study $4$-weighted extremal K\"ahler twins by looking for pairs of the form $(g,1)$ and $(g,f)$. In other words, our viewpoint will be to start with an actual extremal K\"ahler metric $g$ and then look for a non-constant Killing potential $f>0$ such that $(g,1)$ and $(g,f)$ are
$4$-weighted extremal K\"ahler twins.

We will assume that $g$ is extremal K\"ahler and toric, i.e., part of a full toric Delzant structure $(\omega,J,g,T,\mu)$ with a quadrilateral Delzant polytope. From \cite{Leg09} we know that if $g$ is csc then $g$ is either Calabi toric, orthotoric, or a product metric. For this reason, and given our particular interest in csc structures, we will focus this section on those three subcases of toric K\"ahler structures on quadrilaterals. To study extremal toric K\"ahler structures on quadrilaterals in general, one would need to consider all ambitoric structures \cite{ApCaGa15}. We will leave that endeavor to the future or interested readers.

\subsection{Calabi Toric}

Any trapezoid is equivalent (by an affine transform) to a {\it Calabi trapezoid} where

\begin{definition} \label{defnCalabiPolyt}
A \emph{Calabi trapezoid} is a polytope in $\bbr^2$ which is the image of a rectangle $[\alpha_1,\alpha_2]\times[\beta_1,\beta_2]\subset\bbr^2$, with $\alpha_1>0$ and $\beta_1\geq0$, by the map $\sigma : (x,y) \mapsto (x,xy)$. \end{definition}

Let $\Delta$ be a Calabi trapezoid with parameters $\alpha_1,\alpha_2,\beta_1,\beta_2$, with $\alpha_1>0$ and $\beta_1\geq 0$. The normals of $\Delta$ can be written as:
\begin{align} \label{labelCalabi}u_{\alpha_1}=C_{\alpha_1}\begin{pmatrix}
\alpha_1\\
0\end{pmatrix}, \; u_{\alpha_2}=C_{\alpha_2}\begin{pmatrix}
\alpha_2\\
0\end{pmatrix}, \;u_{\beta_1}=C_{\beta_1}\begin{pmatrix}
\beta_1\\
-1\end{pmatrix},\; u_{\beta_2}=C_{\beta_2}\begin{pmatrix}
\beta_2\\
-1\end{pmatrix}\end{align}
with $C_{\alpha_1}$, $C_{\beta_2}>0$ and $C_{\alpha_2}$, $C_{\beta_1}<0$. Thus, any labeled Calabi trapezoid determines and is determined by an $8$--tuple $(\alpha_1,\alpha_2,\beta_1,\beta_2,C_{\alpha_1}, C_{\alpha_2}, C_{\beta_1},C_{\beta_2})$ which we shall refer to as \emph{Calabi parameters}. We shall also refer to the subset $(\alpha_1,\alpha_2,\beta_1,\beta_2)$ as \emph{Calabi coordinates}.

\begin{definition} \label{defnCALABItoric} $(\omega,J,g, T,\mu)$ is a \emph{Calabi toric structure} with \emph{Calabi coordinates} $(x,y) \in [\alpha_1,\alpha_2]\times [\beta_1,\beta_2]$ if $(\mu_1=x,\mu_2=xy,s,t)$ are action-angle coordinates and there exist functions, $A\in C^{\infty}([\alpha_1,\alpha_2])$ and $B\in C^{\infty}([\beta_1,\beta_2])$, such that $A(x)$ and $B(y)$ are positive on the interior and on the interior we have
\begin{align}
g:= g_{A,B}:= x\frac{dx^2}{A(x)} + & x\frac{dy^2}{B(y)}  + \frac{A(x)}{x}(dt + yds)^2 + xB(y)ds^2
\label{CALABItoricmetric}
\end{align} and \begin{align}\label{eq:CondCompactCalabi}
\begin{split}
A(\alpha_i)=0, &\;\; B(\beta_i)=0 \\
A'(\alpha_i) =\frac{2}{C_{\alpha_i}}, &\;\; B'(\beta_i) =-\frac{2}{C_{\beta_i}}.
\end{split}
\end{align}
\end{definition} For a Calabi toric structure $(\omega,J,g_{A,B}, T,\mu)$, the inverse of the Hessian of the symplectic potential is
\begin{align} \bfH_{A,B} = \frac{1}{x}\begin{pmatrix} A(x) & yA(x)\\
yA(x) & x^2B(y) + y^2A(x)
\end{pmatrix},\label{bfHABcalabi}
\end{align} the Laplacian of the momentum map components is 
$$\Delta^{g_{A,B}}\mu_1 =\Delta^{g_{A,B}}x = -\frac{A'(x)}{x}$$ and $$\Delta^{g_{A,B}}\mu_2=\Delta^{g_{A,B}}xy= -\frac{yA'(x)+ xB'(y)}{x},$$ and the scalar curvature of $g_{A,B}$ is \begin{equation}\label{curvatureCALABI}
Scal({g_{A,B}})=-\frac{A''(x)+ B''(y)}{x}.
\end{equation}

We recall the following 
\begin{proposition} \cite{ApCaGa03} \label{calextremal} A Calabi toric structure $(\omega,J, g_{A,B}, T,\mu)$ is extremal (in the Calabi sense, i.e, $Scal({g_{A,B}})$ is affine linear in $(x,xy)$) if and only if $B$ is a polynomial of order $2$ in $y$, $A$ is a polynomial of order at most $4$ in $x$ and $B''(y)=-2A_2$ where    
    \begin{equation}\label{polynACalabi}
A(x) = A_0x^4 +A_1x^3+A_2x^2+A_3x+A_4. 
\end{equation} In this case, 
\begin{equation}\label{curvatureCALABIext}
Scal({g_{A,B}})=-12A_0x - 6A_1
\end{equation} and $(\omega,J, g_{A,B}, T,\mu)$ is K\"ahler-Einstein if and only $A_0=A_3=0$. 
\end{proposition}

\begin{remark} Note that, via the boundary conditions above, each choice of Calabi parameters $(\alpha_1,\alpha_2,\beta_1,\beta_2,C_{\alpha_1}, C_{\alpha_2}, C_{\beta_1},C_{\beta_2})$ with $C_{\beta_1}=-C_{\beta_2}$ uniquely determines $A(x)$ and $B(y)$, polynomials of degree $4$ and $2$ respectively, such that $B''(y)=-2A_2$. Moreover, in that case, $A$ and $B$ are both positive on the interior of $[\alpha_1,\alpha_2]$ and $[\beta_1,\beta_2]$, respectively \cite[Theorem 1.2]{Leg09}.    
\end{remark}

\begin{proposition} If an extremal Calabi toric structure $(\omega,J, g_{A,B}, T,\mu)$ admits a non-constant Killing potential $f>0$ such that
$(g_{A,B},1)$ and $(g_{A,B},f)$ are $4$-weighted extremal K\"ahler twins, then $f$ is of the form $f(x,xy)=\lambda+c_1x$ and $A_3\lambda =2c_1A_4$. In particular, any extremal Calabi toric structure admits at most one extremal Sasaki twin (up to a dilatation) and any K\"ahler--Einstein Calabi toric structure admits no non-trivial twin. 
\end{proposition}
\begin{proof}
We compute the Laplacians assuming $A$ has the form~\eqref{polynACalabi} and $B(y)=-A_2y^2 +B_1y+B_0$. We get
 \begin{align}\label{eq:lap_Calabi}
\begin{split}
\Delta^{g_{A,B}}x &= -\frac{A'(x)}{x}= -\left(4A_0x^2 +3A_1x +2A_2+ \frac{A_3}{x}\right)\\
\Delta^{g_{A,B}}xy &= - \left(4A_0yx^2 +3A_1yx +2A_2y+ \frac{A_3y}{x}\right)  +2A_2y -B_1\\
&= - \left(4A_0yx^2 +3A_1yx + \frac{A_3y}{x}+B_1\right),
\end{split}
\end{align} and writing $f(x,xy)=\lambda + c_1x+c_2xy$ we have 

\begin{align}\label{eq:normGrad_Calabi}
\begin{split}
|df|^2_{g_{A,B}} &=c_1^2  \frac{A(x)}{x}+c_2^2 \left(xB(y) + y^2\frac{A(x)}{x}\right)+ 2c_1c_2y\frac{A(x)}{x}\\
&=y^2\left(c_2^2\left(-A_2x +\frac{A(x)}{x}\right)\right) +(2c_1c_2y + c_1^2)\frac{A(x)}{x}+c_2^2x(B_1y+B_0).   
\end{split}
\end{align} Combining \eqref{curvatureCALABIext}, \eqref{eq:lap_Calabi} and \eqref{eq:normGrad_Calabi} we get that the coefficient of $y^2$ in $f^2Scal(g_{A,B}) -6f\Delta^{g_{A,B}}f -12|df|^2_{g_{A,B}}$ is 

\begin{align}\label{eq:y^2termWeightedScal}
\begin{split}
  & c_2^2x^2(-12A_0x - 6A_1)+6c_2^2( x\left(4A_0x^2 +3A_1x + \frac{A_3}{x}\right) )  - 12c_2^2\left(-A_2x +\frac{A(x)}{x}\right) \\
  &= c_2^2\left( -12A_0x^3 - 6A_1x^2+ 24 A_0x^3 +18A_1x^2 + 6A_3   - 12\left(-A_2x +\frac{A(x)}{x}\right)\right) \\
  &= c_2^2\left( +12A_0x^3  +12A_2x+ 6A_3   - 12\left(-A_2x +\frac{A(x)}{x}\right)\right)\\
  &= c_2^2\left( +12A_0x^3  +12A_2x+ 6A_3   - 12\left( A_0x^3 +A_1x^2 +A_3 +\frac{A_4}{x}\right)\right)\\
  &= c_2^2\left( +12A_0x^3 +12A_1x^2 + 6A_3   - 12A_0x^3- 12A_1x^2 - 12A_3- 12\frac{A_4}{x}\right)\\
  &= c_2^2\left( -6A_3 - 12\frac{A_4}{x}\right).
\end{split}
\end{align} We conclude that {\it if $Scal_{f,4}(g_{A,B})$ is affine linear in $(x,xy)$ then }
$$c_2= 0 \qquad \mbox{ or } \qquad A_3=A_4=0.$$  However, if the latter holds then $$A(x)= x^2(A_0x^2+A_1x+A_2)$$ which is impossible due to the boundary conditions \eqref{eq:CondCompactCalabi}: First, $A$ must have two distinct positive roots; $0<\alpha_1<\alpha_2$ and hence $A_1=-A_0(\alpha_1+\alpha_2)$ and $A_2=A_0\alpha_1\alpha_2$. 
Assuming this, 
$A'(\alpha_1)=-A_0\alpha_1^2(\alpha_2-\alpha_1)$ and $A'(\alpha_2)=A_0\alpha_2^2(\alpha_2-\alpha_1)$ and, since we need $A'(\alpha_1)>0$ and $A'(\alpha_2)<0$, we conclude that $A_0$ must be negative. Since $A_2=A_0\alpha_1\alpha_2$, this in turn forces $A_2$ to be negative and hence $B''(y)=-2A_2 >0$. However, this is not possible due to the boundary conditions on $B$. 

Therefore, if $Scal_{f,4}(g_{A,B})$ is affine linear in $(x,xy)$, then $c_2= 0$. \\ 

Now we assume that $f(x,xy)=\lambda + c_1x$, still assuming that $A(x)$ has the form~\eqref{polynACalabi}, and study the expression for the weighted scalar curvature 
\begin{align}\label{eq:Scalequipoised}
\begin{split}Scal_{f,4}(g_{A,B}) = f^2Scal(g_{A,B}) -6f\Delta^{g_{A,B}}f -12|df|^2_{g_{A,B}}\\
=(\lambda + c_1x)^2 (Scal(g_{A,B}))-6c_1(\lambda + c_1x)\Delta^{g_{A,B}}x -12c_1^2\frac{A(x)}{x}
\end{split}
\end{align} which is rational in $x$ and does not depend on $y$. We find that the coefficients of $x^2$ and of $x^3$ in \eqref{eq:Scalequipoised} are both $0$ while the coefficient of $1/x$ is \begin{equation}\label{eqfWeightCalabicdt}
    6c_1(A_3\lambda -2c_1A_4).
\end{equation} 
Note that the twins are non-colinear if and only if $c_1\neq 0$. This proves the first claim of the proposition. Moreover, we cannot have $A_3=A_4=0$ since this contradicts the boundary condition \eqref{eq:CondCompactCalabi} which implies uniqueness of the twin.

 Finally, if $(\omega,J, g_{A,B}, T,\mu)$ is K\"ahler--Einstein then $A_0=A_3=0$ and by \eqref{eqfWeightCalabicdt} we also have $A_4=0$. This would imply that $A(x)= A_1x^3 + A_2x^2$ which would again contradict the boundary condition \eqref{eq:CondCompactCalabi}.
\end{proof}

\begin{remark}
As we already know from Proposition \ref{page}, the set of extremal Calabi toric metrics, $g$, admitting a non-constant Killing potential $f>0$ such that $(g,1)$ and $(g,f)$ are $4$-weighted extremal K\"ahler twins
is certainly not empty. Below we add yet another example that is somewhat special.
\end{remark}

\begin{example}\label{exemplecscStwins}{\bf CscS twins}
Assume that the Calabi parameters are given as follows:
$$
\begin{array}{cl}&(\alpha_1,\alpha_2,\beta_1,\beta_2,C_{\alpha_1}, C_{\alpha_2}, C_{\beta_1},C_{\beta_2})\\
\\
=&(\alpha_1,\alpha_2,\beta_1,\beta_2,C, -\frac{\alpha_1^2 }{\alpha_2^2}C, -\frac{\alpha_1^2 (\alpha_2-\alpha_1)}{\left(\alpha_1^2+3 \alpha_1 \alpha_2+\alpha_2^2\right) (\beta_2-\beta_1)}C,\frac{\alpha_1^2  (\alpha_2-\alpha_1)}{\left(\alpha_1^2+3 \alpha_1 \alpha_2+\alpha_2^2\right) (\beta_2-\beta_1)}C).
\end{array}
$$
Then $A$ and $B$ from Proposition \ref{calextremal} are given explicitly by
$$A(x)=\frac{2 (x-\alpha_1) (\alpha_2-x) (\alpha_1 x+\alpha_2 x-\alpha_1 \alpha_2)}{\alpha_1^2 (\alpha_2-\alpha_1)C}$$
and 
$$B(y)=\frac{2 \left(\alpha_1^2+3 \alpha_1 \alpha_2+\alpha_2^2\right) (y-\beta_1) (\beta_2-y)}{\alpha_1^2  (\alpha_2-\alpha_1)C}.$$
Since $A$ is a cubic, we have that $Scal(g_{A,B})$ is constant. Indeed,
$$Scal(g_{A,B})=\frac{12 (\alpha_1+\alpha_2)}{\alpha_1^2  (\alpha_2-\alpha_1)C}.$$
If we assume $f(x,xy)=\lambda + c_1x$ with
$\lambda=-1$ and $c_1=\frac{\alpha_1+\alpha_2}{\alpha_1\alpha_2}$, we have that $f$ is positive for $\alpha_1\leq x \leq \alpha_2$, $g$ is $(f,4)$-extremal, and moreover
$$Scal_{f,4}(g_{A,B})=\frac{12 (\alpha_1+\alpha_2) (-\alpha_1 \alpha_2+(\alpha_1+\alpha_2) x)}{\alpha_1^3 \alpha_2  (\alpha_2-\alpha_1)C}= Scal(g_{A,B}) f.$$
As long as the Calabi parameters can occur as a choice of transverse K\"ahler structure for a Sasaki structure this gives a pair of extremal Sasaki twins which are both cscS. 

We have a lot of such examples as we now show. We want to find Calabi parameters satisfying the above and associated to a smooth Sasaki manifold. From \cite[Lemma 4.1]{Leg09}, it is not restrictive to assume that $\alpha_1=1$, $\alpha:=\alpha_2>1$, $\beta_1=0$ and $\beta:=\beta_2>0$. In that case the linear-affine function defining the Calabi polytope are $$\ell_1(u,v)= -C +Cu,\; \; \ell_2(u,v) = CF(\alpha)v, $$ $$\ell_3(u,v) = C-\frac{C}{\alpha}u, \;\; \ell_4(u,v) =CF(\alpha)u  - CF(\alpha)v  $$  where $F(\alpha):=\frac{(\alpha-1)}{\left(1+3 \alpha +\alpha^2\right)}$. By the Delzant--Lerman construction, if the labelled polytope is a transversal polytope of a smooth Sasaki manifold then the symplectic cone over it is, up to an equivariant symplectomorphism, the circle quotient of $\{ z\in \bbc^4 \backslash 0 | \sum_i k_i |z_i|^2 =0 \}$ where $k_i\in \bbz$ are such that $\sum_{i=1}^4 k_i\ell_i =0$, $gcd(k_1,k_2,k_3,k_4)=1$ and the circle action $(\gamma, z) \mapsto  (\gamma^{k_i}z_i)_{i=1,\dots, 4}$ is free. We need to find the $k_i's$. In our situation $\sum_{i=1}^4 k_i\ell_i =0$ implies that $$ -k_1C+k_3 C=0, \;\; k_2CF(\alpha) -k_4CF(\alpha)=0 $$ and $$k_1C -k_3\frac{C}{\alpha} + k_4 CF(\alpha) =0.$$ In particular $k_1=k_3$ and $k_2=k_4$ while $$k_3(1-\alpha) =k_4 \alpha F(\alpha).$$ This last equation (since $\alpha-1>0$) is equivalent to $$-k_3/k_4 = \frac{\alpha}{1+3\alpha +\alpha^2}.$$ 
Now it is easy to see that, as long as $\alpha \in \bbq\cap (1,+\infty)$, we
can choose co-prime solutions $k_3,k_4\in \bbz$. 

\end{example}

\subsection{Orthotoric}

Any generic (no parallel edges) quadrilateral is equivalent to an
{\it orthotoric quadrilateral} where
\begin{definition} \label{defnOrthoPolyt}
A {\it orthotoric quadrilateral} is a quadrilateral in $\bbr^2$ which is
the image of a rectangle $[\alpha_1,\alpha_2]\times[\beta_1,\beta_2]
\subset\bbr^2$, with $\beta_{2} < \alpha_{1}$,
by the map $\sigma : (x,y) \mapsto (x+y,xy)$. \end{definition}

Let $\Delta$ be an orthotoric quadrilateral with parameters
$\alpha_1,\alpha_2,\beta_1,\beta_2$, with $\beta_{2} < \alpha_{1}$.
The normals of $\Delta$ can be written as:
\begin{align} \label{labelortho}u_{\alpha_1}=C_{\alpha_1}\begin{pmatrix}
\alpha_{1}\\
-1\end{pmatrix}, \; u_{\alpha_2}=C_{\alpha_2}\begin{pmatrix}
\alpha_{2}\\
-1\end{pmatrix}, \;u_{\beta_1}=C_{\beta_1}\begin{pmatrix}
\beta_1\\
-1\end{pmatrix},\; u_{\beta_2}=C_{\beta_2}\begin{pmatrix}
\beta_2\\
-1\end{pmatrix}\end{align}
with $C_{\alpha_1}$, $C_{\beta_2}>0$ and $C_{\alpha_2}$, $C_{\beta_1}<0$.
Thus, any labeled orthotoric quadrilateral determines and is determined by a
$8$--tuple
$$(\alpha_1,\alpha_2,\beta_1,\beta_2,C_{\alpha_1}, C_{\alpha_2},
C_{\beta_1},C_{\beta_2})$$ which we shall refer to as {\em orthotoric parameters}.

\begin{definition} \label{defnorthotoric}
$(\omega,J,g, T,\mu)$ is an {\em orthotoric} structure with {\em orthotoric} coordinates $(x,y)\in [\alpha_1,\alpha_2]\times[\beta_1,\beta_2]$, with $\beta_2<\alpha_1$, if \newline
$(\mu_1=x+y,\mu_2=xy,s,t)$ 
are action-angle coordinates
and there exist functions, 
$A\in C^{\infty}([\alpha_1,\alpha_2])$ 
and
$B\in C^{\infty}([\beta_1,\beta_2])$, 
such that $A(x)$ and $B(y)$ are
positive in the interior and in the interior we have
\begin{align}
g= g_{A,B} := \frac{(x-y)}{A(x)}dx^2 + & \frac{(x-y)}{B(y)}dy^2
+ \frac{A(x)}{(x-y)}(dt + yds)^2 + \frac{B(y)}{(x-y)}(dt + xds)^2
\label{orthotoricmetric}
\end{align} and 
\begin{align}\label{eq:CondCompactorthotoric}
\begin{split}
A(\alpha_i)=0, &\;\; B(\beta_i)=0 \\
A'(\alpha_i) =\frac{2}{C_{\alpha_i}}, &\;\; B'(\beta_i) =-\frac{2}{C_{\beta_i}}.
\end{split}
\end{align}
\end{definition}

For an orthotoric structure $(\omega,J,g_{A,B}, T,\mu)$, the inverse of the Hessian of the symplectic potential is
\begin{align} \bfH_{A,B} = \frac{1}{x-y}\begin{pmatrix} A(x)+B(y) &
yA(x)+xB(y)\\
yA(x) +xB(y)&  y^2A(x) +x^2B(y)
\end{pmatrix},\label{bfHABorthotoric}
\end{align} 
the Laplacian of the momentum map components are 
$$\Delta^{g_{A,B}}\mu_1 =\Delta^{g_{A,B}}(x+y) = -\frac{A'(x)+B'(y)}{x-y}$$ and $$\Delta^{g_{A,B}}\mu_2=\Delta^{g_{A,B}}xy= -\frac{yA'(x)+ xB'(y)}{x-y},$$ and the scalar curvature of $g_{A,B}$ is \begin{equation}\label{curvatureOrtho}
Scal({g_{A,B}})=-\frac{A''(x)+ B''(y)}{x-y}.
\end{equation}

We recall the following 
\begin{proposition} \cite{ApCaGa03} \label{Orthoextremal} An orthotoric structure $(\omega,J, g_{A,B}, T,\mu)$ is extremal (in the Calabi sense, i.e, $Scal({g_{A,B}})$ is affine linear in $(x+y,xy)$) if and only if $A$ and $B$ are both polynomials of order at most $4$ and are related as follows:
\begin{equation}\label{polynAOrtho}
A(x) = A_0x^4 +A_1x^3+A_2x^2+A_3x+A_4, 
\end{equation} 
\begin{equation}\label{polynBOrtho}
B(y) = -A_0y^4 -A_1y^3-A_2y^2+B_3y+B_4. 
\end{equation} 
 In this case, 
\begin{equation}\label{curvatureOrthoext}
Scal({g_{A,B}})=-12A_0(x+y) - 6A_1
\end{equation} and $(\omega,J, g_{A,B}, T,\mu)$ is csc if and only if $A_0=0$. Finally,$(\omega,J, g_{A,B}, T,\mu)$ is K\"ahler-Einstein if and only $A_0=0$ and $B_3=-A_3$. 
\end{proposition}

\begin{remark}
From \S 3.2 of \cite{Leg09}, we know that not every choice of orthotoric parameters $(\alpha_1,\alpha_2,\beta_1,\beta_2,C_{\alpha_1}, C_{\alpha_2},
C_{\beta_1},C_{\beta_2})$ allows an extremal structure as above. 
\end{remark}

\begin{proposition} If an extremal orthotoric structure $(\omega,J, g_{A,B}, T,\mu)$ admits a non-constant Killing potential $f>0$ such that
$(g_{A,B},1)$ and $(g_{A,B},f)$ are $4$-weighted extremal K\"ahler twins, then $f$ is of the form $f(x+y,xy)=\lambda+c_1(x+y)$ where $\lambda(A_3+B_3)-2c_1(A_4+B_4)=0$
is a non-trivial condition. In particular, any extremal orthotoric structure admits at most one extremal Sasaki twin (up to a dilatation) and any K\"ahler--Einstein orthotoric structure admits no non-trivial twin. 
\end{proposition}

\begin{proof}
We compute the Laplacians assuming $A$ has the form~\eqref{polynAOrtho} and $B$ has the form~\eqref{polynBOrtho}. We get
\begin{align}\label{eq:lap_Ortho}
\begin{split}
\Delta^{g_{A,B}}(x+y) &= -\frac{A'(x)+B'(y)}{x-y} \\
&= -\frac{4A_0(x^3-y^3)+3A_1(x^2-y^2)+2A_2(x-y)+A_3+B_3}{x-y}\\
\Delta^{g_{A,B}}xy &= -\frac{yA'(x)+ xB'(y)}{x-y} \\
&=-\frac{4A_0xy(x^2-y^2)+3A_1xy(x-y)+A_3y+B_3x}{x-y}\\
\end{split}
\end{align} 
and writing $f(x,xy)=\lambda + c_1(x+y)+c_2xy$ we have 

\begin{align}\label{eq:normGrad_Ortho}
\begin{split}
|df|^2_{g_{A,B}} &=\frac{c_1^2\left(A(x)+B(y)\right)+c_2^2 \left(y^2A(x)+x^2B(y)\right)+ 2c_1c_2\left(yA(x)+xB(y)\right)}{x-y}\\
& =\frac{c_1^2\left(A_0(x^4-y^4)+A_1(x^3-y^3)+A_2(x^2-y^2)+A_3x++B_3 y +A_4+B_4\right)}{x-y}\\
&+\frac{c_2^2 \left(A_0x^2y^2(x^2-y^2)+A_1x^2y^2(x-y)+xy(A_3y+B_3x)+ A_4y^2+B_4x^2\right)}{x-y}\\
&+ \frac{2c_1c_2\left(A_0xy(x^3-y^3)+A_1xy(x^2-y^2)+A_2xy(x-y) +xy(A_3+B_3)+A_4y+B_4x\right)}{x-y}.
\end{split}
\end{align}
Combining \eqref{curvatureOrthoext}, \eqref{eq:lap_Ortho} and \eqref{eq:normGrad_Ortho} we get that $Scal_{f,4}(g_{A,B})=f^2Scal(g_{A,B}) -6f\Delta^{g_{A,B}}f -12|df|^2_{g_{A,B}}$ equals $g(x,y)/(x-y)$ where 
\begin{align}
    \begin{split}
      g(x,y) &= -(\lambda+c_1(x+y)+c_2x y)^2(12A_0(x+y) +6A_1)(x-y)\\
      &+6c_1(\lambda+c_1(x+y)+c_2x y)(4A_0(x^3-y^3)+3A_1(x^2-y^2)+2A_2(x-y)+A_3+B_3)\\
      &+6c_2(\lambda+c_1(x+y)+c_2x y)(4A_0xy(x^2-y^2)+3A_1xy(x-y)+A_3y+B_3x))\\
      &-12c_1^2\left(A_0(x^4-y^4)+A_1(x^3-y^3)+A_2(x^2-y^2)+A_3x+B_3 y +A_4+B_4\right)\\
      &-12c_2^2 \left(A_0x^2y^2(x^2-y^2)+A_1x^2y^2(x-y)+xy(A_3y+B_3x)+ A_4y^2+B_4x^2\right)\\
      &-24c_1c_2\left(A_0xy(x^3-y^3)+A_1xy(x^2-y^2)+A_2xy(x-y) +xy(A_3+B_3)+A_4y+B_4x\right).
    \end{split}
\end{align}
Now, in order for $Scal_{f,4}(g_{A,B})=g(x,y)/(x-y)$ to be an affine linear function in $(x+y,xy)$ we need in particular that $(x-y)$ is a factor of $g(x,y)$. In other words,
$g(x,x)$ has to vanish for all $x$. Now we calculate that
\begin{align}\label{gThatHasToVanishForAllx}
    \begin{split}
      g(x,x) &= 6c_1(\lambda+2c_1x+c_2x^2)(A_3+B_3)\\
      &+6c_2(\lambda+2c_1x+c_2x^2)(A_3+B_3)x\\
      &-12c_1^2\left((A_3+B_3)x +A_4+B_4\right)\\
      &-12c_2^2 \left((A_3+B_3)x^3+ (A_4+B_4)x^2\right)\\
      &-24c_1c_2\left((A_3+B_3)x^2+(A_4+B_4)x\right).
    \end{split}
\end{align}
If we assume that $g_{A,B}$ is K\"ahler-Einstein (so in particular $B_3=-A_3$) this simplifies further to $g(x,x)=-12(A_4+B_4)\left(c_1^2+2c_1c_2x+c_2^2 x^2\right)=-12(A_4+B_4)(c_1+c_2 x)^2$. Since $A_4+B_4=0$ would result in 
$A(t)=-B(t)$ as $4$-degree polynomials, a moment's thought gives that this is just not possible to reconcile with \eqref{eq:CondCompactorthotoric} (since $A(t)$ would have roots at $\beta_1, \beta_2, \alpha_1, \alpha_2$ and also have an impossible fifth root between $\beta_2$ and $\alpha_1$). Therefore, in the K\"ahler-Einstein case, we must have $c_1=c_2=0$ and hence $f$ is just a constant.

Moving forward, we now notice that the coefficient of $x^3$ equals $-6c_2^2(A_3+B_3)$. Since we just saw that $A_3+B_3=0$ is not possible, we then see that we need $c_2=0$ and hence we must have $f=\lambda+c_1(x+y)$. Then
\eqref{gThatHasToVanishForAllx} simplyfies to
\begin{align}\label{gThatHasToVanishForAllx2}
    \begin{split}
      g(x,x) &= 6c_1(\lambda+2c_1x)(A_3+B_3)\\
      &-12c_1^2\left((A_3+B_3)x +A_4+B_4\right)\\
      &= 6c_1(\lambda(A_3+B_3)-2c_1(A_4+B_4)).
    \end{split}
\end{align}
In the non-trivial case ($c_1\neq 0$) we then get the necessary equation
$\lambda(A_3+B_3)-2c_1(A_4+B_4)=0$. Since the above discussion showed that
we cannot have $(A_3+B_3)=(A_4+B_4)=0$, the equation $\lambda(A_3+B_3)-2c_1(A_4+B_4)=0$ is never trivial, and thus, up to an overall re-scale, there is at most one choice of non-trivial $f$ making $Scal_{f,4}(g_{A,B})$ an affine function in $(x+y,xy)$. 
\end{proof}

\begin{remark}
It is not hard (albeit a bit tedious) to check that if $\lambda(A_3+B_3)-2c_1(A_4+B_4))=0$
is satisfied then $g(x,y)/(x-y)$ is indeed an affine linear function in
$(x+y,xy)$ as desired. So assuming we have chosen our orthotoric parameters $(\alpha_1,\alpha_2,\beta_1,\beta_2,C_{\alpha_1}, C_{\alpha_2},
C_{\beta_1},C_{\beta_2})$ in a way that we have an orthotoric extremal structure
$g_{A,B}$ and assuming that $f=\lambda+c_1(x+y)$ with $\lambda(A_3+B_3)-2c_1(A_4+B_4)=0$ is such that it is positive
for $\alpha_1\leq x\leq \alpha_2$, $\beta_1\leq y \leq \beta_2$, then we have an actual example.
\end{remark}

\subsection{Product Toric}
Any parallelogram is equivalent (by an affine transformation) to a rectangle
$[\alpha_1,\alpha_2]\times[\beta_1,\beta_2]\subset\bbr^2$, with
$\alpha_2>\alpha_1$ and $\beta_2>\beta_1$.

Let $\Delta$ be the rectangle $[\alpha_1,\alpha_2]\times[\beta_1,\beta_2]\subset\bbr^2$. The normals of $\Delta$ can be written as:
\begin{align} \label{labelParallelogram}u_{\alpha_1}=C_{\alpha_1}\begin{pmatrix}
1\\
0\end{pmatrix}, \; u_{\alpha_2}=C_{\alpha_2}\begin{pmatrix}
1\\
0\end{pmatrix}, \;u_{\beta_1}=C_{\beta_1}\begin{pmatrix}
0\\
-1\end{pmatrix},\; u_{\beta_2}=C_{\beta_2}\begin{pmatrix}
0\\
-1\end{pmatrix}\end{align}
with $C_{\alpha_1}$, $C_{\beta_2}>0$ and $C_{\alpha_2}$, $C_{\beta_1}<0$.
Thus, any labeled rectangle determines and is determined by an
$8$--tuple
$(\alpha_1,\alpha_2,\beta_1,\beta_2,C_{\alpha_1}, C_{\alpha_2},
C_{\beta_1},C_{\beta_2})$ which we shall refer to as {\em product toric parameters}.

\begin{definition} \label{defnparallel}
The quintuple $(\omega,J,g, T,\mu)$ is a {\em product toric} structure with {\em product toric} coordinates $(x,y)\in [\alpha_1,\alpha_2]\times[\beta_1,\beta_2]$ if \newline
$(\mu_1=x,\mu_2=y,s,t)$ 
are action-angle coordinates
and there exist functions, 
$A\in C^{\infty}([\alpha_1,\alpha_2])$ 
and
$B\in C^{\infty}([\beta_1,\beta_2])$, 
such that $A(x)$ and $B(y)$ are
positive in the interior and in the interior
\begin{align}
g= g_{A,B} := \frac{dx^2}{A(x)} + & \frac{dy^2}{B(y)}  + A(x)dt^2 + B(y)ds^2 \label{PRODtoricmetric}
\end{align} and 
\begin{align}\label{eq:CondCompactPRODtoric}
\begin{split}
A(\alpha_i)=0, &\;\; B(\beta_i)=0 \\
A'(\alpha_i) =\frac{2}{C_{\alpha_i}}, &\;\; B'(\beta_i) =-\frac{2}{C_{\beta_i}}.
\end{split}
\end{align}
\end{definition}

For a product toric structure $(\omega,J,g, T,\mu)$, it is straightforward to see that
the Laplacian of the momentum map components are 
$$\Delta^{g_{A,B}}\mu_1 =\Delta^{g_{A,B}}x = -A'(x)$$ and $$\Delta^{g_{A,B}}\mu_2=\Delta^{g_{A,B}}y= -B'(y),$$ and the scalar curvature of $g_{A,B}$ is 
\begin{equation}\label{curvaturePROD}
Scal({g_{A,B}})=-(A''(x)+ B''(y)).
\end{equation}

We can now state the following simple companion to Propositions \ref{calextremal} and \ref{Orthoextremal}:
\begin{proposition} \cite{ApCaGa03} \label{PRODextremal} A product toric structure $(\omega,J, g_{A,B}, T,\mu)$ is extremal (in the Calabi sense, i.e, $Scal({g_{A,B}})$ is affine linear in $(x,y)$) if and only if $A$ and $B$ are both polynomials of order at most $3$:
\begin{equation}\label{polynAPROD}
A(x) = A_0x^3+A_1x^2+A_2x+A_3. 
\end{equation} 
\begin{equation}\label{polynBPROD}
B(y) = B_0y^3+B_1y^2+B_2y+B_3. 
\end{equation} 
 In that case, 
\begin{equation}\label{curvaturePRODext}
Scal({g_{A,B}})=-6(A_0x+B_0y) - 2(A_1+B_1)
\end{equation} and $(\omega,J, g_{A,B}, T,\mu)$ has constant scalar curvature if and only if $A_0=B_0=0$. 
\end{proposition}

\begin{proposition}\label{extprodprop}
If an extremal product toric structure $(\omega,J, g_{A,B}, T,\mu)$ admits a non-constant Killing potential $f>0$ such that
$(g_{A,B},1)$ and $(g_{A,B},f)$ are $4$-weighted extremal K\"ahler twins, then $f$ is of the form $f(x,y)=\lambda+c_1 x + c_2 y$ where $A_0c_2=B_0c_1$ and
\begin{itemize}
    \item if $c_1\neq 0$ (and hence $A_0\neq 0$),
then $3A_0 \lambda =c_1(A_1+B_1)$
    \item if $c_2\neq 0$, (and hence $B_0\neq 0$)
 then $3B_0\lambda  =c_2(A_1+B_1)$.    
\end{itemize}
In particular, any extremal product toric structure admits at most one extremal Sasaki twin (up to a dilatation) and any csc product toric structure admits no non-trivial twin. 
\end{proposition}

\begin{proof}
We compute the Laplacians assuming $A$ has the form~\eqref{polynAPROD} and $B$ has the form~\eqref{polynBPROD}. We get
\begin{align}\label{eq:lap_PROD}
\begin{split}
\Delta^{g_{A,B}}x &= -A'(x)= -3A_0 x^2-2A_1 x-A_2\\
\Delta^{g_{A,B}}y &= -B'(x)= -3B_0 y^2-2B_1 y-B_2\\
\end{split}
\end{align} 
and writing $f(x,xy)=\lambda + c_1 x+c_2y$ we have 
\begin{align}\label{eq:normGrad_PROD}
\begin{split}
|df|^2_{g_{A,B}} &= c_1^2 A(x)+c_2^2 B(x)\\
&= c_1^2\left(A_0x^3+A_1x^2+A_2x+A_3\right)+c_2^2\left(B_0y^3+B_1y^2+B_2y+B_3\right).
\end{split}
\end{align}

Combining \eqref{curvaturePRODext}, \eqref{eq:lap_PROD} and \eqref{eq:normGrad_PROD} we get that $Scal_{f,4}(g_{A,B})=f^2Scal(g_{A,B}) -6f\Delta^{g_{A,B}}f -12|df|^2_{g_{A,B}}$ equals 
$g(x,y)$ where
\begin{align}
    \begin{split}
      g(x,y) &= -(\lambda+c_1x+c_2 y)^2(-6(A_0x+B_0y) - 2(A_1+B_1))\\
      &+6c_1(\lambda+c_1x+c_2 y)(3A_0 x^2+2A_1 x+A_2)\\
      &+6c_2(\lambda+c_1x+c_2 y)(3B_0 y^2+2B_1 y+B_2)\\
      &-12c_1^2\left(A_0x^3+A_1x^2+A_2x+A_3\right)\\
      &-12c_2^2 \left(B_0y^3+B_1y^2+B_2y+B_3\right)\\
      &= -2 (\lambda^2(A_1+B_1) - 3 \lambda(A_2 c_1 + B_2  c_2) + 6 (A_3 c_1^2 + 
   B_3 c_2^2))\\
   &-2 (3 A_0 \lambda^2 - 4 A_1 \lambda c_1 + 2 B_1 \lambda c_1 + 3 A_2 c_1^2 - 3 B_2 c_1 c_2) x\\
   &-2 (3 B_0 \lambda^2 - 4 B_1 \lambda c_2 + 2 A_1 \lambda c_2 + 3 B_2 c_2^2 - 3 A_2 c_1 c_2) y\\
   &+2 c_1 (3 A_0 \lambda - c_1(A_1+B_1)) x^2\\
   &+2 c_2 (3 B_0 \lambda - c_2(A_1+B_1)) y^2\\
   &-4 (3 \lambda( B_0  c_1 +  A_0  c_2) - 2 (A_1+B_1) c_1 c_2) x y \\
   &+6 c_1 (A_0 c_2-B_0 c_1) x^2 y\\
   &+6 c_2 (B_0 c_1-A_0 c_2) x y^2 .
    \end{split}
\end{align}
This is affine linear in $(x,y)$ if and only if the coefficients of
$x^2, y^2, xy, x^2y$ and $xy^2$ all vanish.
If we are in the csc case ($A_0=B_0=0$) then, looking at the coefficients of
$x^2$ and $y^2$, we need
$$c_1^2(A_1+B_1)=0\quad\text{and}\quad c_2^2(A_1+B_1)=0.$$
In the non-trivial case where at least one of $c_1,c_2$ does not vanish, this gives $A_1=-B_1$ and since $A_0=B_0$ this would mean that one of $A(x)$ and $B(y)$ is a concave up parabola or both $A(x)$ and $B(y)$ are linear. Neither is compatible with \eqref{eq:CondCompactPRODtoric}. This proves the last part of the proposition.

Moving forward we assume we are in the non-trivial case, i.e., at least one of $c_1$ and $c_2$ is non-zero. In that case, $g(x,y)$ is affine linear in $(x,y)$, if and only if
\begin{align}\label{solvingPROD}
     \begin{split}
 A_0 c_2-B_0 c_1&=0\\
3 \lambda( B_0  c_1 +  A_0  c_2) - 2 (A_1+B_1) c_1 c_2&=0\\
c_1 (3 A_0 \lambda - c_1(A_1+B_1))&=0\\
c_2 (3 B_0 \lambda - c_2(A_1+B_1))&=0    
     \end{split}
\end{align}

Assume that we choose $c_1\neq 0$. Then the first equation in
\eqref{solvingPROD} together with the fact that not both $A_0$ and $B_0$ can be zero, tells us that $A_0\neq 0$. Further, the system \eqref{solvingPROD} is
equivalent to the non-trivial system
\begin{align}\label{solvingPROD2}
     \begin{split}
 A_0c_2&=B_0c_1\\
3A_0 \lambda &=c_1(A_1+B_1).\\    
     \end{split}
\end{align}

Likewise, if $c_2\neq 0$, then $B_0\neq 0$ and 
the system \eqref{solvingPROD} is
equivalent to
\begin{align}\label{solvingPROD3}
     \begin{split}
  A_0c_2&=B_0c_1\\
3B_0\lambda  &=c_2(A_1+B_1).\\    
     \end{split}
\end{align}
This proves the first part of the proposition.
\end{proof}

The last part of Proposition \ref{extprodprop} gives

\begin{corollary}\label{SEonS3xS3hasnotwin}
    The regular csc Sasaki structure arising from the Boothby-Wang construction over a product csc K\"ahler metric on $\bbc\bbp^1\times \bbc\bbp^1$ has no
    extremal Sasaki twin. 
\end{corollary}

\def\cprime{$'$} \def\cprime{$'$} \def\cprime{$'$} \def\cprime{$'$}
  \def\cprime{$'$} \def\cprime{$'$} \def\cprime{$'$} \def\cdprime{$''$}
  \def\cprime{$'$} \def\cprime{$'$} \def\cprime{$'$} \def\cprime{$'$}
\providecommand{\bysame}{\leavevmode\hbox to3em{\hrulefill}\thinspace}
\providecommand{\MR}{\relax\ifhmode\unskip\space\fi MR }
\providecommand{\MRhref}[2]{%
  \href{http://www.ams.org/mathscinet-getitem?mr=#1}{#2}
}
\providecommand{\href}[2]{#2}


\begin{thebibliography}{ACGTF08}

\bibitem[Abr01]{Abr01}
M.~Abreu, \emph{K\"ahler metrics on toric orbifolds}, J. Differential Geom.
  \textbf{58} (2001), no.~1, 151--187. \MR{1895351 (2003b:53046)}

\bibitem[Abr10]{Abr09}
Miguel Abreu, \emph{K\"ahler-{S}asaki geometry of toric symplectic cones in
  action-angle coordinates}, Port. Math. \textbf{67} (2010), no.~2, 121--153.
  \MR{2662864}

\bibitem[AC21]{ApCa18}
Vestislav Apostolov and David M.~J. Calderbank, \emph{The {CR} geometry of
  weighted extremal {K}\"{a}hler and {S}asaki metrics}, Math. Ann. \textbf{379}
  (2021), no.~3-4, 1047--1088. \MR{4238260}

\bibitem[ACG03]{ApCaGa03}
V.~Apostolov, D.~M.~J. Calderbank, and P.~Gauduchon, \emph{The geometry of
  weakly self-dual {K}\"ahler surfaces}, Compositio Math. \textbf{135} (2003),
  no.~3, 279--322. \MR{1956815 (2004f:53045)}

\bibitem[ACG06]{ApCaGa06}
Vestislav Apostolov, David M.~J. Calderbank, and Paul Gauduchon,
  \emph{Hamiltonian 2-forms in {K}\"ahler geometry. {I}. {G}eneral theory}, J.
  Differential Geom. \textbf{73} (2006), no.~3, 359--412. \MR{2228318
  (2007b:53149)}

\bibitem[ACG15]{ApCaGa15}
\bysame, \emph{Ambitoric geometry {II}: extremal toric surfaces and {E}instein
  4-orbifolds}, Ann. Sci. \'Ec. Norm. Sup\'er. (4) \textbf{48} (2015), no.~5,
  1075--1112. \MR{3429476}

\bibitem[ACG16]{ApCaGa16}
\bysame, \emph{Ambitoric geometry {I}: {E}instein metrics and extremal
  ambik\"ahler structures}, J. Reine Angew. Math. \textbf{721} (2016),
  109--147. \MR{3574879}

\bibitem[ACGTF04]{ACGT04}
V.~Apostolov, D.~M.~J. Calderbank, P.~Gauduchon, and C.~W.
  T{\o}nnesen-Friedman, \emph{Hamiltonian $2$-forms in {K}\"ahler geometry.
  {II}. {G}lobal classification}, J. Differential Geom. \textbf{68} (2004),
  no.~2, 277--345. \MR{2144249}

\bibitem[ACGTF08]{ACGT08}
Vestislav Apostolov, David M.~J. Calderbank, Paul Gauduchon, and Christina~W.
  T{\o}nnesen-Friedman, \emph{Hamiltonian 2-forms in {K}\"ahler geometry.
  {III}. {E}xtremal metrics and stability}, Invent. Math. \textbf{173} (2008),
  no.~3, 547--601. \MR{2425136 (2009m:32043)}

\bibitem[AJL23]{ApJuLa21}
Vestislav Apostolov, Simon Jubert, and Abdellah Lahdili, \emph{Weighted
  {K}-stability and coercivity with applications to extremal {K}\"ahler and
  {S}asaki metrics}, Geom. Topol. \textbf{27} (2023), no.~8, 3229--3302.
  \MR{4668097}

\bibitem[AM19]{ApMa19}
Vestislav Apostolov and Gideon Maschler, \emph{Conformally {K}\"{a}hler,
  {E}instein-{M}axwell geometry}, J. Eur. Math. Soc. (JEMS) \textbf{21} (2019),
  no.~5, 1319--1360. \MR{3941493}

\bibitem[AMTF22]{ApMaTF18}
Vestislav Apostolov, Gideon Maschler, and Christina~W. T{\o}nnesen-Friedman,
  \emph{Weighted extremal {K}\"{a}hler metrics and the {E}instein-{M}axwell
  geometry of projective bundles}, Comm. Anal. Geom. \textbf{30} (2022), no.~4,
  689--744. \MR{4545849}

\bibitem[Apo22]{Apostolov19}
Vestislav Apostolov, \emph{The {K}\"ahler geometry of toric manifolds},
  preprint;arXiv:2208.12493 [math.DG] (2022).

\bibitem[BG00]{BG00b}
Charles~P. Boyer and Krzysztof Galicki, \emph{A note on toric contact
  geometry}, J. Geom. Phys. \textbf{35} (2000), no.~4, 288--298. \MR{1780757
  (2001h:53124)}

\bibitem[BG08]{BG05}
\bysame, \emph{Sasakian geometry}, Oxford Mathematical Monographs, Oxford
  University Press, Oxford, 2008. \MR{MR2382957 (2009c:53058)}

\bibitem[BGS08]{BGS06}
Charles~P. Boyer, Krzysztof Galicki, and Santiago~R. Simanca, \emph{Canonical
  {S}asakian metrics}, Commun. Math. Phys. \textbf{279} (2008), no.~3,
  705--733. \MR{2386725}

\bibitem[BHLTF17]{BHLT15}
Charles~P. Boyer, Hongnian Huang, Eveline Legendre, and Christina~W.
  T{\o}nnesen-Friedman, \emph{The {E}instein-{H}ilbert functional and the
  {S}asaki-{F}utaki invariant}, Int. Math. Res. Not. IMRN (2017), no.~7,
  1942--1974. \MR{3658189}

\bibitem[BHLTF18]{BHLT16}
\bysame, \emph{Reducibility in {S}asakian geometry}, Trans. Amer. Math. Soc.
  \textbf{370} (2018), no.~10, 6825--6869. \MR{3841834}

\bibitem[BHLTF21]{BHLT21}
\bysame, \emph{Some open problems in {S}asaki geometry}, Differential geometry
  in the large, London Math. Soc. Lecture Note Ser., vol. 463, Cambridge Univ.
  Press, Cambridge, 2021, pp.~143--168. \MR{4420789}

\bibitem[BHLTF23]{BHLT23}
\bysame, \emph{Existence and non-existence of constant scalar curvature and
  extremal {S}asaki metrics}, Math. Z. \textbf{304} (2023), no.~4, Paper No.
  61, 29. \MR{4617164}

\bibitem[BP14]{BoPa10}
Charles~P. Boyer and Justin Pati, \emph{On the equivalence problem for toric
  contact structures on {$S^3$}-bundles over {$S^2$}}, Pacific J. Math.
  \textbf{267} (2014), no.~2, 277--324. \MR{3207586}

\bibitem[Bry01]{Bry01}
R.~L. Bryant, \emph{Bochner-{K}\"ahler metrics}, J. Amer. Math. Soc.
  \textbf{14} (2001), no.~3, 623--715 (electronic). \MR{1824987 (2002i:53096)}

\bibitem[BTF14]{BoTo13}
Charles~P. Boyer and Christina~W. T{\o}nnesen-Friedman, \emph{Extremal
  {S}asakian geometry on {$S^3$}-bundles over {R}iemann surfaces}, Int. Math.
  Res. Not. IMRN (2014), no.~20, 5510--5562. \MR{3271180}

\bibitem[BTF16]{BoTo14a}
\bysame, \emph{The {S}asaki join, {H}amiltonian 2-forms, and constant scalar
  curvature}, J. Geom. Anal. \textbf{26} (2016), no.~2, 1023--1060.
  \MR{3472828}

\bibitem[BTF24]{BTF23}
\bysame, \emph{Sasakian geometry on sphere bundles {II}: {C}onstant scalar
  curvature}, Differential Geom. Appl. \textbf{96} (2024), Paper No. 102166,
  23. \MR{4761312}

\bibitem[BW58]{BoWa58}
W.~M. Boothby and H.~C. Wang, \emph{On contact manifolds}, Ann. of Math. (2)
  \textbf{68} (1958), 721--734. \MR{22 \#3015}

\bibitem[Cal82]{cal82}
E.~Calabi, \emph{Extremal {K}\"ahler metrics}, Seminar on Differential
  Geometry, Ann. of Math. Stud., vol. 102, Princeton Univ. Press, Princeton,
  N.J., 1982, pp.~259--290. \MR{83i:53088}

\bibitem[Der83]{Derd83}
Andrzej Derdzi\'{n}ski, \emph{Self-dual {K}\"{a}hler manifolds and {E}instein
  manifolds of dimension four}, Compositio Math. \textbf{49} (1983), no.~3,
  405--433. \MR{707181}

\bibitem[FO18]{FutOno18}
Akito Futaki and Hajime Ono, \emph{Volume minimization and conformally
  {K}\"ahler, {E}instein-{M}axwell geometry}, J. Math. Soc. Japan \textbf{70}
  (2018), no.~4, 1493--1521. \MR{3868215}

\bibitem[FO19]{FutOno19}
\bysame, \emph{Conformally {E}instein-{M}axwell {K}\"ahler metrics and
  structure of the automorphism group}, Math. Z. \textbf{292} (2019), no.~1-2,
  571--589. \MR{3968916}

\bibitem[KTF16]{KoTF16}
Caner Koca and Christina~W. T{\o}nnesen-Friedman, \emph{Strongly {H}ermitian
  {E}instein-{M}axwell solutions on ruled surfaces}, Ann. Global Anal. Geom.
  \textbf{50} (2016), no.~1, 29--46. \MR{3521556}

\bibitem[Lah19]{Lah19}
Abdellah Lahdili, \emph{K\"{a}hler metrics with constant weighted scalar
  curvature and weighted {K}-stability}, Proc. Lond. Math. Soc. (3)
  \textbf{119} (2019), no.~4, 1065--1114. \MR{3964827}

\bibitem[LeB10]{Leb10}
Claude LeBrun, \emph{The {E}instein-{M}axwell equations, extremal {K}\"{a}hler
  metrics, and {S}eiberg-{W}itten theory}, The many facets of geometry, Oxford
  Univ. Press, Oxford, 2010, pp.~17--33. \MR{2681684}

\bibitem[LeB15]{Leb15}
\bysame, \emph{The {E}instein-{M}axwell equations, {K}\"{a}hler metrics, and
  {H}ermitian geometry}, J. Geom. Phys. \textbf{91} (2015), 163--171.
  \MR{3327057}

\bibitem[LeB16]{Leb16}
\bysame, \emph{The {E}instein-{M}axwell equations and conformally {K}\"{a}hler
  geometry}, Comm. Math. Phys. \textbf{344} (2016), no.~2, 621--653.
  \MR{3500251}

\bibitem[Lee86]{lee86}
J.M Lee, \emph{The {F}efferman metric and pseudo-{H}ermitian invariants},
  Transactions of the American Mathematical Society (1986), no.~1, 412--429.

\bibitem[Leg11a]{Leg10}
Eveline Legendre, \emph{Existence and non-uniqueness of constant scalar
  curvature toric {S}asaki metrics}, Compos. Math. \textbf{147} (2011), no.~5,
  1613--1634. \MR{2834736}

\bibitem[Leg11b]{Leg09}
\bysame, \emph{Toric geometry of convex quadrilaterals}, J. Symplectic Geom.
  \textbf{9} (2011), no.~3, 343--385. \MR{2817779 (2012g:53079)}

\bibitem[Ler03]{Ler02a}
Eugene Lerman, \emph{Contact toric manifolds}, J. Symplectic Geom. \textbf{1}
  (2003), no.~4, 785--828. \MR{2039164}

\bibitem[Pag78]{Pag}
Don Page, \emph{A compact rotating gravitational instanton}, Phys. Lett. B
  (1978), 235--238.

\bibitem[Tan89]{tanno89}
S.~Tanno, \emph{Variational problems on contact riemannian manifolds},
  Transactions of the American Mathematical Society (1989), no.~134, 349--379.

\bibitem[TF98]{To-Fr98}
Christina~Wiis T{\o}nnesen-Friedman, \emph{Extremal {K}\"ahler metrics on
  minimal ruled surfaces}, J. Reine Angew. Math. \textbf{502} (1998), 175--197.
  \MR{MR1647571 (99g:58026)}

\bibitem[VdS21]{VizaDS21}
Isaque Viza~de Souza, \emph{Conformally {K}\"ahler, {E}instein-{M}axwell
  metrics on {H}irzebruch surfaces}, Ann. Global Anal. Geom. \textbf{59}
  (2021), no.~2, 263--284. \MR{4213767}

\end{thebibliography}
\end{document}